\documentclass[a4paper,12pt]{article}
\usepackage[cp1251]{inputenc}            
\usepackage[english]{babel}
\usepackage[left=2.5cm,right=1.5cm,top=2cm,bottom=2cm]{geometry} 
\usepackage{amsmath,amssymb,amsthm}
\usepackage{envmath}
\usepackage{mathrsfs}
\usepackage{caption2}
\usepackage{floatflt}
\usepackage[multiple]{footmisc}

\usepackage{lipsum}
\usepackage{authblk}
\usepackage{fancyhdr}

\newtheorem{theorem}{Theorem}
\newtheorem{lemma}{Lemma}

\newtheorem{proposition}{Proposition}

\newtheorem{corollary}{Corollary}

\newtheorem{remark}{Remark}

\renewenvironment{proof}[1][Proof]{\begin{trivlist}
\item[\hskip \labelsep {\bfseries {#1.}}]}{\end{trivlist}}

\addto\captionsrussian{}

\begin{document}







\title{Sequential two-fold Pearson chi-squared test and tails of the Bessel process distributions}

\author{Savelov M.P.}
\affil{Lomonosov Moscow State University\\Moscow Institute of Physics and Technology\\
Email: savelov.mp@mipt.ru}

\maketitle

\begin{abstract}
We find asymptotic formulas for error probabilities of two-fold Pearson goodness-of-fit test as functions of two critical levels. These results may be reformulated in terms of tails of two-dimensional distributions of the Bessel process. Necessary properties of the Infeld function are obtained.
\end{abstract}

\section{Introduction}

In this paper we study joint distributions of the values of Pearson statistics arising in the sequential $\chi^2 $ test (see \cite{Sevastyanov}, \cite{Selivanov}-\cite{Tymanyan}). Consider a scheme of independent trials with $N$ outcomes ($N \ge 2$). The hypothesis $H$ states that the probability of the $j$-th outcome is $p_j, j = 1,...,N $. Let $n_1<n_2<... <n_r $ be the volumes of nested samples. We consider the variables $\mu_j^{(k)} $ which are equal to 1 if $j$-th outcome appeared in the $k$-th trial and $\mu_j^{(k)} = 0 $ in another case. Denote by $\nu_{ij}: = \sum_{k=1}^{n_i} \mu_j^{(k)} $ the number of $j$ -th outcomes in the first $n_i$ trials. We consider the Pearson statistics $X(n_i):=\sum_{j=1}^N \frac{(\nu_{ij} - n_i p_j)^2}{n_i p_j}$. Define $X:=( X(n_1), X(n_2), ..., X(n_r))$. A sequential $r$-fold Pearson chi-squared test is constructed with the help of statistics $X$ as follows: the critical values $x_1^*, ..., x_r^*$ are selected; the hypothesis $H$ is rejected if and only if $X(n_k) > x_k^*$ for all $k=1,2,...,r$, i.e. the probability $\alpha$ of rejection of the basic hypothesis $H$ in case when it is true is equal to $ P(X(n_1) > x_1^*, X(n_2) > x_2^*, ..., X(n_r) > x_r^* | H)$. In other words, $\alpha$ is the level of significance of the test.

In \cite{Sevastyanov} the following formula for $\alpha$ was obtained (see (4.10)):
\begin{eqnarray}
\alpha = \alpha(x_1^*, ..., x_r^*)= \int_{\frac{x_1^*}2}^{\infty}...\int_{\frac{x_r^*}2}^{\infty} p_{r,0}(u_1, ..., u_r)du_1...du_r \nonumber
\end{eqnarray}
where
\begin{gather}
p_{r,0}(u_1, ...,u_r) =\frac{e^{-\sum_{i=1}^r \frac{u_i}{\lambda_i} } }{K_r^{1+\delta} (b_1 \cdot ... \cdot b_{r-1})^{1+\delta} \Gamma(1+\delta) \Pi_{i=1}^r  \lambda_i }  \cdot \( \frac{u_1 u_r}{\lambda_1 \lambda_r}\)^{\frac{\delta}2} \cdot \Pi_{i=1}^{r-1} I_{\delta} \( 2 \sqrt{b_i \cdot \frac{u_i u_{i+1}}{\lambda_i \lambda_{i+1}} } \). \label{5STATbigformulaoshib}
\end{gather}
The parameters $K_r$, $\theta$, $\lambda_i$, $b_i$ are expressed in terms of $n_i$, $N$ and $r$ as follows:
\begin{gather}
\rho_0: =0, \  \rho_r := 0, \rho_i: = \sqrt{\frac{n_i}{n_{i+1}}}, \ b_i:= \frac{\rho_i^2 (1-\rho_{i-1}^2)(1-\rho_{i+1}^2)}{(1-\rho_{i-1}^2\rho_i^2)(1-\rho_i^2 \rho_{i+1}^2)}, \nonumber\\
\lambda_i = \frac{(1-\rho_{i-1}^2)(1-\rho_i^2)}{1-\rho_{i-1}^2 \rho_i^2}, K_r:=\frac{\prod_{k=1}^{r-1}(1-\rho_k^2 \rho_{k-1}^2)}{\prod_{k=1}^{r-1}(1-\rho_k^2)}, \ \delta = \frac{N-3}2. \nonumber
\end{gather}
Here $I_{\delta}(x)$ is the Infeld function (see, e.g., \cite{lebedev}) which is defined as follows: $I_{\delta}(x):=\sum_{k = 0}^{\infty} \frac{\(\frac{x}{2}\)^{\delta + 2k}}{\Gamma(k+1)\Gamma(k+\delta+1)}$.

Moreover, in the formula (\ref{5STATbigformulaoshib})  (see (4.10) in \cite{Sevastyanov}) there is a typo, which is easy to establish by going from formula (4.7) to formula (4.8), in which the degree of the product $(b_1 \cdots b_{r-1})$ in the denominator is incorrect and it led to the error in (4.10).

Thus, the correct formula for $p_{r,0}$ has the following form:
\begin{gather}
p_{r,0}(u_1, ...,u_r) =\frac{e^{-\sum_{i=1}^r \frac{u_i}{\lambda_i} } }{K_r^{1+\delta} (b_1 \cdot ... \cdot b_{r-1})^{\frac{\delta}2} \Gamma(1+\delta) \Pi_{i=1}^r  \lambda_i }  \cdot \( \frac{u_1 u_r}{\lambda_1 \lambda_r}\)^{\frac{\delta}2} \cdot \Pi_{i=1}^{r-1} I_{\delta} \( 2 \sqrt{b_i \cdot \frac{u_i u_{i+1}}{\lambda_i \lambda_{i+1}} } \). \label{5STATbigformula}
\end{gather}


In this paper we will consider only the case $ r = 2 $. Let us introduce the following notation.
By definition, let $c:=\sqrt{\frac{n_1}{n_2}}$, $K_2:= \frac{1}{1-c^2}$, $\beta:= 1 - c^2$, $\delta: = \frac{N-3}2$. It is easy to see that $0 < c < 1$.
In addition, we assume (following \cite{Sevastyanov}) that $n_1, n_2 \to \infty$ so that $c$ converges to some limiting value, which we also denote by $c$ (assume that $ 0 <c <1 $).

In this notation (and in case $r=2$) the formula (\ref{5STATbigformula}) transforms to the following equality:\begin{eqnarray}
\alpha = \alpha(x_1^*, x_2^*) = \int_{\frac{x_1^*}2}^{\infty}\int_{\frac{x_2^*}2}^{\infty} p_{2,0}(u_1, u_2)du_1 du_2 \nonumber
\end{eqnarray}
where
$$p_{2,0}(u_1,u_2)=\frac{e^{- (\frac{u_1}{\beta} + \frac{u_2}{\beta}) } }{K_2^{1+\delta} c^{\delta} \Gamma(1+\delta) \beta^2 }  \cdot \( \frac{u_1 u_2}{\beta^2}\)^{\frac{\delta}2} \cdot I_{\delta} \( \frac{2c\sqrt{ u_1 u_2 }}{\beta} \).$$

Proceeding from this formula, we obtain asymptotic formulas for the significance level $\alpha $ (in the case $ r = 2 $) and also estimate the error of this asymptotic formula. Note that it allows us (with the help of Bonferroni inequality) to obtain two-sided estimates of $ \alpha $ in the case of an arbitrary $r$.

In order to obtain explicit formulas for $ \alpha $ we require the properties of the Infeld function $I_{\nu}(x)$, which are presented in the following theorem and are of independent interest. Let
\begin{gather}
\Psi(\nu, x):=e^{-2x} + \left( \frac{|4\nu^2 - 1|}{8x}  + G(\nu, x) \cdot \frac{|(4\nu^2 - 1)(4 \nu^2-9)|}{32x^2} \right) \cdot (1 + e^{-2x})\nonumber
\end{gather}
where $G(\nu, x) = \( 1 - \frac{\nu - \frac12}{2 x}\)^{ - 2\nu - 1}$ if $\nu \ge \frac12$ and
$G(\nu, x) =      \biggl( 1 - \frac{\nu+\frac32}{2x}\biggr)^{-\nu-\frac32} \cdot \( 1 + \frac{2 \nu + 2}{x}\)$ if
$0 \le \nu < \frac12$. Obviously, the function $\Psi(\nu, x)$ decreases monotonically as $x > 0$.
The following theorem holds true.
\begin{theorem} \label{5STATth1}
If $\nu > \frac12$ and $2x_0 > \nu - \frac12$ or if $0 \le \nu < \frac12$ and $2x_0 > \nu + \frac32$ then the inequality $\Biggl|  \frac{I_{\nu}(x)}{\left( \frac{e^x}{\sqrt{2 \pi x}} \right)} -1 \Biggr| \le \Psi(\nu, x)$ holds true for all $x \ge x_0$.
If  $\nu = \frac12$ then $\Biggl|  \frac{I_{\nu}(x)}{\left( \frac{e^x}{\sqrt{2 \pi x}} \right)} -1 \Biggr| \le \Psi(\nu, x) := e^{-2x}$ for all $x > 0$.
\end{theorem}
Next, we need the following notation: ${\rho}:=\sqrt{\frac{{x}_2^*}{{x}_1^*}}$ and ${\lambda} := {\frac{x_1^*}{2\beta}} $.  The following proposition holds true.
\begin{proposition} \label{5STATpredl1}
Let $N \ge 3$, ${\lambda}>1$,
\begin{gather}
x_1^*x_2^* > \frac{(1-c^2)^2}{c^2}\cdot \( \(\frac{N}4 -1 \)^2\cdot I(N > 4) +\(\frac{N}4\)^2\cdot I(N = 3) \), \nonumber
\end{gather}
and $c<\sqrt{\frac{{x}_2^*}{{x}_1^*}} < \frac1{c}$. Then the set of numbers  $\varepsilon$ satisfying the following conditions :
$$0 < \varepsilon < \min\( \frac{{\rho}}{c}-1, \frac1{c}-\rho \), \ \ \ \varepsilon(2\min({\rho}-c, 1-c{\rho})+\varepsilon) <  \frac{1-c^2}{c^2} - ({\rho}^2 - 2c{\rho} + 1)$$
is nonempty and for each of such $\varepsilon$ the following relation holds:
\begin{gather}
\alpha(x_1^*, x_2^*) =  \frac{(x_1^* x_2^*)^{\frac{N}4}}{(2c)^{\frac{N}{2}-1}\Gamma(\frac{N-1}2)\sqrt{ \pi(1-c^2)}} \cdot  I_{2} (1+ \theta_1) \nonumber
\end{gather}
where
\begin{gather}
I_{2} = \int_{ 1 }^{\infty} dt_1 \int_{  1}^{\infty}  dt_2 \cdot e^{- {\lambda}(t_1^2- 2c {\rho} t_1 t_2 + {\rho}^2 t_2^2)  } \cdot (t_1 t_2)^{\delta + \frac12} = \nonumber \\
= e^{-{\lambda}({\rho}^2-2c{\rho}+1)} \cdot \(\frac{(1+\theta_{2})(1-\theta_{3})(1-\theta_{4})(1+\theta_{5})(1-\theta_{6})(1-\theta_{7})}{4{\lambda}^2 {\rho} ({\rho}-c)(1-c{\rho})} + \tilde{I}_{4}\) \nonumber
\end{gather}
and
\begin{gather}
0 \le \tilde{I}_{4}  \le e^{- {\lambda}\varepsilon(2\min({\rho}-c, 1-c{\rho})+\varepsilon)} \cdot
e^{({\rho}^2 - 2c {\rho} +1 + \varepsilon(2\min({\rho}-c, 1-c{\rho})+\varepsilon))}I_{6}, \nonumber \\
0 \le I_{6} \le \frac1{{\rho}^{\frac{N}2} } \cdot \((2c)^{\frac{N}2-1} \sqrt{\pi} \cdot \frac{\Gamma(\frac{N-1}2)}{2(1-c^2)^{\frac{N-1}2}} + 2^{\frac{N}2-2}  \cdot \frac{\(\Gamma(\frac{N}4)\)^2}{2(1-c^2)^{\frac{N}4 }}
 \), \nonumber \\
|\theta_1| \le \Psi\(\frac{N-3}2, \frac{c\sqrt{x_1^*x_2^*} }{1-c^2}\), \ \ \
0 \le \theta_{2}  \le ((1+\varepsilon)(1+\frac{\varepsilon}{{\rho}}))^{\frac{N}2-1}-1, \ \ \
0 \le \theta_{3} \le \frac{\varepsilon}{{\rho} - c}, \nonumber \\
0 \le \theta_{4} \le e^{-{\lambda}\varepsilon(2({\rho} - c - c \varepsilon)+ \varepsilon)}, \ \ \ 0 \le \theta_{5} \le \frac{c\varepsilon}{{\rho}-c-c\varepsilon},\nonumber \\
0 \le \theta_{6} \le \frac{\varepsilon}{1-c{\rho}+\varepsilon}, \ \ \theta_{7} =e^{-{\lambda} \varepsilon(2(1-c{\rho}) + \varepsilon)}. \nonumber
\end{gather}
\end{proposition}

\begin{theorem} \label{5STATth2}
Let  $N \ge 3$, ${x_1^*} \to
+\infty$, ${x_2^*} \to
+\infty$ so that ${\rho} = \sqrt{\frac{{x}_2^*}{{x}_1^*}}$ is fixed and $c<{\rho} < \frac1{c}$. Then
\begin{gather}
\alpha = \frac{(x_1^*)^{\frac{N}2-2} {\(\frac{{\rho}}{2c}\)}^{\frac{N}2-1} \cdot (1-c^2)^{\frac32} \cdot e^{-\frac{x_1^*}{2(1-c^2)}({\rho}^2-2c{\rho}+1)}}{\sqrt{\pi}\Gamma(\frac{N-1}2)({\rho}-c)(1-c{\rho})} \cdot \(1 + O\(\frac{\ln x_1^*}{x_1^*}\)\). \nonumber
\end{gather}
\end{theorem}
Further, the following theorem holds true.
\begin{theorem} \label{5STATth3}
Suppose that $N \ge 3$, $\alpha_1:=\lim_{n_1 \to \infty}P(X(n_1)>x_1^*)$, $\alpha_2:=\lim_{n_2 \to \infty}P(X(n_2)>x_2^*)$  and, as before, $\alpha = \lim_{n_i \to \infty, \frac{n_1}{n_2} \to c^2}P(X(n_1)>x_1^*, X(n_2)>x_2^*)$. Suppose that $x_1^* \to +\infty$ and  $x_2^* \to +\infty$ so that $\sqrt{\frac{\ln \alpha_2}{\ln \alpha_1}} = const =:P$. If  $c < P <\frac1{c}$ then
\begin{gather}
\alpha \sim \frac{(1-c^2)^{\frac{3}2}\cdot P^{\frac{N}2-1}\cdot (- \ln \alpha_1)^{\frac{N}2-2} \cdot Q^{-\frac1{2(1-c^2)}}}{2c^{\frac{N}2-1}\sqrt{\pi}\Gamma\(\frac{N-1}2\)(P-c)(1-c P)}    \nonumber
\end{gather}
where
\begin{gather}
Q=\frac{P^{2(N-3)(1-\frac{c}{P})} }{ \(\Gamma\(\frac{N-1}2\) \)^{4(1-Pc)} } \cdot (-\ln \alpha_1)^{(N-3)(2-c(P+P^{-1}))} \cdot  \alpha_1^{-2(P^2-2Pc+1)} .
 \nonumber
\end{gather}
\end{theorem}
\begin{corollary}  \label{5STATsled1}
If the condititions of the previous theorem holds true and $\alpha_2 = \alpha_1$ (i.e. $P=1$) then
\begin{gather}
\alpha \sim
\frac{(1-c^2)^{\frac32}{\(\Gamma(\frac{N-1}2)\)}^{\frac{1-c}{1+c}}}{2c^{\frac{N}2-1}\sqrt{\pi}(1-c)^2} \cdot
\(- \ln \alpha_1\)^{\frac{N}2 - \frac{N-3}{1+c}-2} \(\alpha_1 \)^{\frac{2}{1+c}}
\nonumber
\end{gather}
where $\frac{2}{1+c} \in (1,2)$ since $c \in (0,1)$.
\end{corollary}

\begin{remark}
Note that (under the conditions of Corollary \ref{5STATsled1}, when $\alpha_1=\alpha_2$ and hence  $x_1^* = x_2^*$) if $c$ <<is close>> to 1, which corresponds to the case when $n_1$ <<is close>> to $n_2$ and sample of size $n_2$ <<does not differ much>> from a sample of size $n_1$, then it is natural to expect that the value  $P(\chi_1^2 > x_1^*, \chi_2^2 > x_2^*| H)$ is <<close>> to $\alpha_1$  in order.  In the case when $c$ <<is close>> to 0, which corresponds to the case when $n_1$ <<differs a lot>> from $n_2$ and the sample of size $n_2$ <<differs a lot>> from the sample of size $n_1$, it is natural to expect (according to <<almost independence>> of $\chi_1^2$ and $\chi_2^2$) that the value  $P(\chi_1^2 > x_1^*, \chi_2^2 > x_2^*| H)$ is <<close>> to $\alpha_1^2$ in order. These observations fit well with the formula from Corollary \ref{5STATsled1}.
\end{remark}

Finally, the following theorem holds true.
\begin{theorem} \label{5STATth4}
Suppose $d \ge 2$ and $Bes_d(t)$ is $d$-dimensional Bessel process, i.e. euclidean norm of the $d$-dimensional Brownian motion. Suppose $0 < s_1 < s_2$ and $x_1, x_2\to +\infty$ so that $\sqrt{\frac{s_1}{s_2}}\frac{x_2}{x_1} = \rho = const$ and $1 < \frac{x_2}{x_1} < \frac{s_2}{s_1}$. Then for any function $g(x)$ such that $g(x) \to + \infty$ as $x \to +\infty$ we have
\begin{gather}
P(Bes_{d}(s_1) \ge x_1, Bes_{d}(s_2) \ge x_2) = \nonumber \\
 =\frac{(x_1^*)^{\frac{d-3}2} {\(\frac{{\rho}}{2c}\)}^{\frac{d-1}2} \cdot (1-c^2)^{\frac32} \cdot e^{-\frac{x_1^*}{2(1-c^2)}({\rho}^2-2c{\rho}+1)}}{\sqrt{\pi}\Gamma(\frac{d}2)({\rho}-c)(1-c{\rho})} \cdot \(1 + O\(\frac{\ln x_1^*}{x_1^*}\)\) \nonumber
\end{gather}
where $x_1^* :=\frac{x_1^2}{s_1}$ and $c:=\sqrt{\frac{s_1}{s_2}}$.
\end{theorem}

\section{Proof of Theorem \ref{5STATth1} }
We need the following representation of the Bessel function $J_{\nu}(z)$ (
see \cite{lebedev})  via Bessel functions of the third kind. Let $z \in { \mathbb{C} \setminus \{ z: ( \mathrm{Im}(z) = 0) \& (\mathrm{Re}(z) \le 0) \} }  $. Then $J_{\nu} (z) = \frac{H_{\nu}^{(1)} (z) + H_{\nu}^{(2)}(z) }2$ (see \cite{lebedev}).
Further, define
\begin{gather}(\nu, m): = \frac{(4 {\nu}^2-1)\cdot(4 {\nu}^2 - 3^2)\cdots (4 \nu^2 - (2m-1)^2)}{2^{2m} \cdot m!} = \frac{\Gamma(\nu+m+\frac12)}{m!\Gamma(\nu-m+\frac12)}. \nonumber
\end{gather}
By definition (see \cite{watson}) we put
\begin{eqnarray}
\sum {}_{\nu}^{(1)} (z;p):= \sum_{m=0}^{p-1} \frac{(-1)^m \cdot (\nu, m)}{(2iz)^m}, \nonumber\\
\sum {}_{\nu}^{(2)} (z;p):= \sum_{m=0}^{p-1} \frac{(\nu, m)}{(2iz)^m}. \nonumber
\end{eqnarray}
In addition, let now $\nu > \frac12, \mathrm{Re}(z) \ge 0, r:=|z|$ and let $2r \ge \nu - \frac12$. Define $G:=\( 1- \frac{\nu-\frac12}{2r} \)^{-\nu-\frac12}$. The Weber's formulas for the remainders in expansions of Bessel functions of the third kind (for $p \ge 1$) have the following form (see \cite{watson}):
\begin{eqnarray}
H_{\nu}^{(1)} (z)  = \( \frac{2}{\pi z} \)^{\frac12} \cdot e^{i(z-\frac{\pi \nu}{2} - \frac{\pi}4 )}  \cdot \(\sum {}_{\nu}^{(1)} (z;p) + R_p^{(1)} \),  \nonumber \\
H_{\nu}^{(2)} (z) =\( \frac{2}{\pi z} \)^{\frac12} \cdot e^{-i(z-\frac{\pi \nu}{2} - \frac{\pi}4 )}  \cdot \(\sum {}_{\nu}^{(2)} (z;p) + R_p^{(2)} \)  \nonumber
\end{eqnarray}
and the branch of the function $z^{\frac12}$ is chosen in accordance with the condition  $Re(z) \ge 0$ (i.e. $\arg(z) \in [-\frac{\pi}2, \frac{\pi}2]$) and in this case
\begin{eqnarray}
|R_p^{(1)}|, |R_p^{(2)}|  \le 2\cdot G^2 \cdot |(\nu, p)| \cdot \frac{\Gamma(\frac12) \cdot \Gamma(\frac{p}2 + 1)}{\Gamma(\frac{p+1}2) \cdot |(2z)^p|}. \nonumber
\end{eqnarray}

Further, for $-\pi < \arg(z) < \frac{\pi}2$ we have $I_{\nu}(z) = e^{- \frac{i \pi \nu}{2}} \cdot J_{\nu}(z \cdot e^{\frac{i \pi}2})$ (see \cite{lebedev}). We need only the case $x >0$, $\nu \ge 0$. Starting from the formulas presented above we obtain convenient expressions for the remainder of the series in the expansion of the function $I_{\nu}(x)$  in powers of $x$. Everywhere in what follows we assume that $x >0$, $2x > \nu - \frac12$, $\nu \ge 0$. Consistently expressing $I_{\nu}(z)$ through $J_{\nu}(z)$ and $J_{\nu}(z)$ through $H_{\nu}^{(1)} (z)$ and $H_{\nu}^{(2)} (z)$ we obtain the following relations:
\begin{gather}
I_{\nu}(x) = e^{- \frac{i \pi \nu}{2}} \cdot J_{\nu}(ix) = \nonumber \\
=e^{- \frac{i \pi \nu}{2}} \cdot \frac12 \cdot \( \frac{2}{\pi i x} \)^{\frac12}  \left(
 e^{i(ix-\frac{\pi \nu}{2} - \frac{\pi}4 )}  \(\sum {}_{\nu}^{(1)} (ix;p) + \tilde{R}_p^{(1)} \) +
e^{-i(ix-\frac{\pi \nu}{2} - \frac{\pi}4 )} \(\sum {}_{\nu}^{(2)} (ix;p) + \tilde{R}_p^{(2)} \)
 \right) \nonumber
 \end{gather}
where $(\frac1{i})^{\frac12} = e^{-\frac{i \pi}4}$. Since $r = |ix| = x$ we have: $G=\( 1- \frac{\nu-\frac12}{2x} \)^{-\nu-\frac12}$, \begin{eqnarray}
|\tilde{R}_p^{(1)}|, |\tilde{R}_p^{(2)}|  \le 2\cdot G^2 \cdot |(\nu, p)| \cdot \frac{\Gamma(\frac12) \cdot \Gamma(\frac{p}2 + 1)}{\Gamma(\frac{p+1}2) \cdot |(2 x)^p|}. \nonumber
\end{eqnarray} Hence, we obtain that
\begin{gather}
I_{\nu}(x) =  e^{- \frac{i \pi \nu}{2}} e^{-\frac{i \pi}4} \frac{1}{\sqrt{2 \pi x}}  \cdot D \nonumber \\ \text{where \ \ }
D :=
e^{-x} \cdot e^{-i (\frac{\pi \nu}2 + \frac{\pi}4)}\(\sum {}_{\nu}^{(1)} (ix;p) + \tilde{R}_p^{(1)} \) + e^{x} \cdot e^{i (\frac{\pi \nu }2 + \frac{\pi}4 )} \(\sum {}_{\nu}^{(2)} (ix;p) + \tilde{R}_p^{(2)} \) . \nonumber
\end{gather}
Denote $A_p: = \sum {}_{\nu}^{(1)} (ix;p)$, $B_p: = \sum {}_{\nu}^{(2)} (ix;p)$. Thus
$A_p = \sum_{m=0}^{p-1} \frac{(-1)^m \cdot (\nu, m)}{(-2x)^m} = \sum_{m=0}^{p-1} \frac{(\nu, m)}{(2x)^m}$,  $B_p = \sum_{m=0}^{p-1} \frac{(\nu, m)}{(-2x)^m} = \sum_{m=0}^{p-1} \frac{(-1)^m \cdot (\nu, m)}{(2x)^m}$. It follows that
\begin{gather}
I_{\nu}(x) = \frac{1}{\sqrt{2 \pi x}} \cdot \left(  e^{-x} \cdot e^{-i (\pi \nu + \frac{\pi}2 )}  \cdot (A_p + \tilde{R}_p^{(1)})  +  e^{x} \cdot (B_p + \tilde{R}_p^{(2)}) \right)
. \nonumber
\end{gather}
Thus, we have proved the following proposition.
\begin{proposition} \label{5STATpredl2}
Let $2x > \nu - \frac12>0$. Let $A_p = \sum_{m=0}^{p-1} \frac{(\nu, m)}{(2x)^m}$,  $B_p = \sum_{m=0}^{p-1} \frac{(-1)^m \cdot (\nu, m)}{(2x)^m}$, $G = \( 1- \frac{\nu-\frac12}{2x} \)^{-\nu-\frac12}$,  $(\nu, m) = \frac{(4 {\nu}^2-1)\cdot(4 {\nu}^2 - 3^2)\cdots (4 \nu^2 - (2m-1)^2)}{2^{2m} \cdot m!}$. Then the following relations hold:
\begin{gather}
I_{\nu}(x) = \frac{1}{\sqrt{2 \pi x}} \cdot \left(  e^{-x} \cdot e^{-i (\pi \nu + \frac{\pi}2 )}  \cdot (A_p + \tilde{R}_p^{(1)})  +  e^{x} \cdot (B_p + \tilde{R}_p^{(2)}) \right), \nonumber \\
|\tilde{R}_p^{(1)}|, |\tilde{R}_p^{(2)}|   \le 2\cdot G^2 \cdot |(\nu, p)| \cdot \frac{\Gamma(\frac12) \cdot \Gamma(\frac{p}2 + 1)}{\Gamma(\frac{p+1}2) \cdot |(2 x)^p|}. \nonumber
\end{gather}
\end{proposition}

\begin{remark} \label{5STATzame4oBessel}

Note that in the case $\nu \in [0, \frac12)$ the same formulas will be fulfilled if we denote by $G$ an
expression $\biggl( 1 - \frac{\nu+\frac32}{2r}\biggr)^{-\nu-\frac32} \cdot (1+ \frac{2\nu+2}{r})$ and require that $2r > \nu + \frac32$ (see \cite{watson}).  If $\nu = \frac12$ then (see \cite{lebedev})   $I_{\frac12}(x) =\bigl(\frac{2}{\pi x}\bigr)^{\frac12}\cdot {\rm{sh}}  x$, which makes it easy to get similar formulas in this case.  Finally, if $\nu < 0$ then the recurrent formula $I_{\nu-1}(z)-I_{\nu+1}(z)=\frac{2\nu}{z}I_{\nu}(z)$  (see \cite{watson}) allows us to find the asymptotics of $I_{\nu}(z)$ for negative $\nu$ by reducing the finding of this asymptotics to the case $\nu \ge 0$,  which is analyzed above.

\end{remark}

In what follows we use only the first two terms ($p=2$) of the expansion from Proposition \ref{5STATpredl2}. However, the arguments are completely analogous in the case of $p >2$. Now we pass directly to the proof of Theorem  \ref{5STATth1}.  We carry out the proof for the case $\nu > \frac12$ (by Remark \ref{5STATzame4oBessel} the case $0 \le \nu < \frac12$ is completely analogous and the case $\nu = \frac12$  is trivial).
In fact, for  $p=2$ we obtain that $A_2 = 1 + \frac{4\nu^2 - 1}{2^2 \cdot 1! \cdot (2x)^1}$, $B_2 = 1 + \frac{(-1)\cdot(4\nu^2 - 1)}{2^2 \cdot 1! \cdot (2x)^1}$,
\begin{gather}
I_{\nu}(x) = \frac{1}{\sqrt{2 \pi x}} \cdot \left(  e^{-x} \cdot e^{-i (\pi \nu + \frac{\pi}2 )}  \cdot (1 + \frac{4\nu^2 - 1}{2^2 \cdot 1! \cdot (2x)^1} + \tilde{R}_2^{(1)})  +  e^{x} \cdot ( 1 + \frac{(-1)\cdot(4\nu^2 - 1)}{2^2 \cdot 1! \cdot (2x)^1} + \tilde{R}_2^{(2)}) \right), \nonumber \\
|\tilde{R}_2^{(1)}|, |\tilde{R}_2^{(2)}|   \le 2\cdot G^2 \cdot \frac{|(4\nu^2 - 1)(4 \nu^2-9)|}{2^4 \cdot 2!} \cdot \frac{\Gamma(\frac12) \cdot \Gamma(2)}{\Gamma(\frac{3}2) \cdot |(2 x)^2|} = \nonumber \\
=  \frac{G^2}{16} \cdot |(4\nu^2 - 1)(4 \nu^2-9)| \cdot  \frac{\Gamma(\frac12) \cdot 1!}{\frac12 \Gamma(\frac{1}2) \cdot 4x^2} = \frac{G^2}{32 x^2} \cdot |(4\nu^2 - 1)(4 \nu^2-9)| =: \Delta. \nonumber
\end{gather}
Therefore,
\begin{gather}
\frac{I_{\nu}(x)}{\left( \frac{e^x}{\sqrt{2 \pi x}} \right) } - 1  = \left(  e^{-2x} \cdot e^{-i (\pi \nu + \frac{\pi}2 )}  \cdot (1 + \frac{4\nu^2 - 1}{8x} + \tilde{R}_2^{(1)})  +   \frac{(-1)\cdot(4\nu^2 - 1)}{8x} + \tilde{R}_2^{(2)} \right) \nonumber
\end{gather}
hence
\begin{gather}
\Biggl|  \frac{I_{\nu}(x)}{\left( \frac{e^x}{\sqrt{2 \pi x}} \right) } - 1 \Biggr| \le e^{-2x} \cdot \bigl| 1 + \frac{4\nu^2 - 1}{8x} \bigr|  + \Bigl| \frac{(4\nu^2 - 1)}{8x} \Bigr| + \Delta + \Delta \cdot (e^{-2x}) = \nonumber \\
= e^{-2x} + \bigl| \frac{(4\nu^2 - 1)}{8x} \bigr| \cdot (1 + e^{-2x}) + \Delta \cdot (1 + e^{-2x}) = e^{-2x} + \left( \frac{|4\nu^2 - 1|}{8x}  + \Delta \right) \cdot (1 + e^{-2x}).\nonumber
\end{gather}
In this case $G = \( 1 - \frac{\nu - \frac12}{2 x}\)^{ - \nu - \frac12}$. Thus $\Delta = \frac{G^2}{32x^2} \cdot |(4\nu^2 - 1)(4 \nu^2-9)| = \( 1 - \frac{\nu - \frac12}{2 x}\)^{ - 2\nu - 1} \cdot \frac{|(4\nu^2 - 1)(4 \nu^2-9)|}{32x^2}$.
Hence we obtain that
\begin{gather}
\Biggl|  \frac{I_{\nu}(x)}{\left( \frac{e^x}{\sqrt{2 \pi x}} \right) } - 1 \Biggr|  \le e^{-2x} + \left( \frac{|4\nu^2 - 1|}{8x}  + \Delta \right) \cdot (1 + e^{-2x}) \le \nonumber  \\
\le e^{-2x} + \left( \frac{|4\nu^2 - 1|}{8x}  + \( 1 - \frac{\nu - \frac12}{2 x}\)^{ - 2\nu - 1} \cdot \frac{|(4\nu^2 - 1)(4 \nu^2-9)|}{32x^2} \right) \cdot (1 + e^{-2x}) = \Psi(\nu, x).\nonumber
\end{gather} Theorem \ref{5STATth1} is proved.

\section{Proof of Proposition \ref{5STATpredl1} and Theorem \ref{5STATth2}}
First we prove Proposition  \ref{5STATpredl1}. We recall the formula for  $\alpha$ established in \cite{Sevastyanov}:
\begin{eqnarray}
\alpha= \int_{\frac{x_1^*}2}^{\infty} \int_{\frac{x_2^*}2}^{\infty} p_{2,0}(u_1, u_2)du_1 du_2 \nonumber
\end{eqnarray}
where
$$p_{2,0}(u_1,u_2)=\frac{e^{- (\frac{u_1}{\beta} + \frac{u_2}{\beta}) } }{K_2^{1+\delta} c^{\delta} \Gamma(1+\delta) \beta^2 }  \cdot \( \frac{u_1 u_2}{\beta^2}\)^{\frac{\delta}2} \cdot I_{\delta} \( \frac{2 c \sqrt{u_1 u_2} }{\beta} \). $$
Let us rewrite it in the following form:
\begin{gather}
\alpha = C \int_{\frac{x_1^*}2}^{\infty} \int_{\frac{x_2^*}2}^{\infty}  e^{- \frac{u_1}{\beta} - \frac{u_2}{\beta} } \cdot \(\frac{u_1 u_2}{\beta^2}\)^{\frac{\delta}2} \cdot I_{\delta} \( \frac{2 c\sqrt{u_1 u_2}}{\beta} \) \frac{du_1 du_2}{\beta^2}  \nonumber
\end{gather}
where $C = \frac{1}{K_2^{1+\delta} c^{\delta}\Gamma(1+\delta)} = \frac{c}{(K_2 c)^{1+\delta} \Gamma(1+\delta)}$. We make the substitution $t_1:=\sqrt{\frac{u_1}{\beta}}$, $t_2:=\sqrt{\frac{u_2}{\beta}}$. The Jacobian of the substitution is  $4t_1 t_2 \beta^2$, therefore
\begin{gather}
\alpha = 4C \int_{ x_1 }^{\infty} \int_{ x_2 }^{\infty}  e^{- t_1^2- t_2^2 } \cdot (t_1 t_2)^{\delta + 1} \cdot I_{\delta} \( 2c t_1 t_2 \) dt_1 dt_2 \nonumber
\end{gather}
where  $x_1 = \sqrt{\frac{x_1^*}{2\beta}}$ and $x_2 = \sqrt{ \frac{x_2^*}{2 \beta}}$.
Further, if $\delta > \frac12$ then in order to use Theorem  \ref{5STATth1} we require that the following condition holds: $ 2c\sqrt{\frac{x_1^* x_2^*}{4\beta^2}}> \frac{\delta}2 - \frac14$. In the case $\delta \in [0, \frac12)$ we will require that  $ 2c \sqrt{\frac{x_1^* x_2^*}{4\beta^2 }}> \frac{\delta}2 + \frac34$. In the case $\delta = \frac12$ nothing is required. In other words, we require that the following conditions hold:  $\delta \ge 0$,
\begin{gather}
x_1^*x_2^* > \frac{(1-c^2)^2}{c^2}\cdot \( \(\frac{\delta}2-\frac14\)^2\cdot I(\delta > \frac12) +\(\frac{\delta}2+\frac34\)^2\cdot I(0 \le \delta < \frac12) \) \nonumber
\end{gather}
where $I(A)=1$ or $I(A)=0$ if the condition $ A $ is satisfied or not satisfied respectively.
These conditions are equivalent to the fact that  $N \ge 3$ and
\begin{gather}
x_1^*x_2^* > \frac{(1-c^2)^2}{c^2}\cdot \( \(\frac{N}4 -1 \)^2\cdot I(N > 4) +\(\frac{N}4\)^2\cdot I(N = 3) \). \nonumber
\end{gather}
By Theorem \ref{5STATth1} and the monotonicity of the function $\Psi(\nu, x)$ we have:
\begin{gather}
\alpha  = 4C (1 + \theta_1) \int_{ x_1 }^{\infty} \int_{  x_2 }^{\infty}  e^{- t_1^2- t_2^2 } \cdot (t_1 t_2)^{\delta + 1} \cdot \frac{e^{2c t_1 t_2} }{\sqrt{2 \pi} \sqrt{ 2c t_1 t_2} } dt_1 dt_2 \nonumber
\end{gather}
where, as before, $|\theta_1| \le \Psi(\delta, 2c x_1 x_2) = \Psi\(\delta,
\frac{c\sqrt{ x_1^*x_2^*} }{1-c^2}\)$.
Hence,
\begin{equation}
\alpha = \frac{2C}{\sqrt{\pi c}} (1 + \theta_1) I_{1} \label{5STATalphaI1}
\end{equation} where \begin{gather}
I_{1} :=
\int_{ x_1 }^{\infty} \int_{  x_2 }^{\infty}  e^{- (t_1^2- 2c t_1 t_2 + t_2^2) } \cdot (t_1 t_2)^{\delta + \frac12}
dt_1 dt_2. \nonumber
\end{gather}
By definition, put $\rho=\frac{x_2}{x_1}$ (whence $\rho = \sqrt{\frac{x_2^*}{x_1^*}} > 0$) and change the variables in the last integral: $u = x_1^{-1} t_1$, $v= x_2^{-1} t_2$ (where $x_1$ and $x_2$  are fixed and defined above), then
\begin{gather}
I_{1} = \int_{ 1 }^{\infty} x_1 du \int_{  1}^{\infty} x_2dv \cdot e^{- (x_1^2 u^2- 2 c x_1 u x_2 v + x_2^2 v^2)  } \cdot (x_1 u x_2 v)^{\delta + \frac12}
= (x_1 x_2)^{(\delta+\frac{3}2 )} \cdot I_{2} \label{5STATI17I18}
\end{gather}
where \begin{gather}
I_{2} = \int_{ 1 }^{\infty} dt_1 \int_{  1}^{\infty}  dt_2 \cdot e^{- \lambda(t_1^2- 2c \rho t_1 t_2 + {\rho}^2 t_2^2)  } \cdot (t_1 t_2)^{\delta + \frac12} \nonumber
\end{gather}
and where $\lambda := x_1^2$. Let us obtain the asymptotics of the Laplace integral from the previous formula with the help of standard methods of analysis. \\
Suppose that $c<\frac{x_2}{x_1} < \frac1{c}$ (i.e. $c < \rho < \frac1{c}$). For an arbitrary  $\varepsilon > 0$ we put by definition \\ ${D}_{\varepsilon}:=\{(t_1, t_2): t_1 \ge 1, t_2 \ge 1, \max((t_1-1), {\rho} (t_2 - 1)) \ge \varepsilon \}$. Then
\begin{equation}
I_{2} = I_{3}+I_{4} \label{5STATI18I19I20}
\end{equation}
where
\begin{gather}
I_{3}:= \int_{1}^{1+\varepsilon}  dt_1 \int_{1}^{1+\frac{\varepsilon}{{\rho}}} dt_2 \cdot e^{-  {\lambda}(t_1^2- 2c{\rho} t_1 t_2 + {\rho}^2 t_2^2)  } \cdot (t_1 t_2)^{\delta + \frac12}, \nonumber \\
I_{4}:= \int_{{D}_{\varepsilon}} dt_1 dt_2 \cdot e^{- {\lambda}(t_1^2- 2c {\rho} t_1 t_2 + {\rho}^2 t_2^2)  } \cdot (t_1 t_2)^{\delta + \frac12}. \nonumber
\end{gather}
Let us obtain convenient formulas for $I_{3}$ and $I_{4}$. Note that
\begin{gather}
I_{3}= (1+\theta_{2})\int_{1}^{1+\varepsilon}  dt_1 \int_{1}^{1+\frac{\varepsilon}{{\rho}}} dt_2 \cdot e^{- {\lambda}(t_1^2- 2c{\rho} t_1 t_2 + {\rho}^2 t_2^2)  } \nonumber
\end{gather}
where $0 \le \theta_{2}  \le ((1+\varepsilon)(1+\frac{\varepsilon}{{\rho}}))^{\delta+\frac12}-1$. We make the substitution: $u=t_1 - 1, v=\rho(t_2-1)$ (whence $t_1=u+1$ and $t_2 = \frac{v}{\rho}+1$) and get:
\begin{gather}
I_{3}= (1+\theta_{2})\int_{0}^{\varepsilon}  du \int_{0}^{\varepsilon} \frac{dv}{\rho} \cdot e^{- \lambda((u+1)^2- 2c \rho (u+1) (\frac{v}{\rho}+1) + \rho^2 (\frac{v}{\rho}+1)^2)  }= \nonumber\\
=(1+\theta_{2})\int_{0}^{\varepsilon}  du \int_{0}^{\varepsilon} \frac{dv}{\rho} \cdot e^{- \lambda(u^2 +2u+1- 2cuv -2cv -2c \rho u -2c\rho + v^2 + 2v\rho + \rho^2)}=\nonumber \\
 =(1+\theta_{2})\cdot \frac{e^{-\lambda(\rho^2-2c\rho+1)}}{\rho} \cdot  \int_{0}^{\varepsilon}  du \cdot
 e^{- \lambda(u^2 +2u(1-c\rho))} \int_{0}^{\varepsilon} dv \cdot e^{- \lambda(v^2 + 2(\rho-c-cu)v)}= \nonumber\\
 =(1+\theta_{2})\cdot \frac{e^{-\lambda(\rho^2-2c\rho+1)}}{\rho} \cdot  \int_{0}^{\varepsilon}  du \cdot
 e^{- \lambda(u^2 +2u(1-c\rho) - (\rho-c-cu)^2)} \int_{0}^{\varepsilon} dv \cdot e^{- \lambda(v^2 + 2(\rho-c-cu)v + (\rho-c-cu)^2)}=\nonumber \\
 =(1+\theta_{2})\cdot \frac{e^{-\lambda(\rho^2-2c\rho+1)}}{\rho} \cdot  \int_{0}^{\varepsilon}  du \cdot
 e^{- \lambda(u^2 +2u(1-c\rho) - (\rho-c-cu)^2)} \cdot F(u) \label{5STATI12I14}
\end{gather}
where
\begin{gather}
F(u):=\int_{0}^{\varepsilon} dv \cdot e^{-\lambda(v+b_u)^2}, \ \ \ b_u:= \rho-c-cu. \nonumber
\end{gather}
If $0 \le u \le \varepsilon$ then $b_u \in [\rho-c-c\varepsilon, \rho -c]$. Let us assume that $\rho-c-c\varepsilon>0$  (i.e. $\varepsilon < \frac{\rho}{c}-1$). In the expression for $F(u)$  we put $w:=(v+b_u)^2$ (whence $v=\sqrt{w}-b_u$) and get:
\begin{gather}
F(u)=\int_{b_u^2}^{(b_u+\varepsilon)^2}e^{-\lambda w} \cdot \frac{dw}{2\sqrt{w}}=\frac{1}{2b_u}\int_{b_u^2}^{(b_u+\varepsilon)^2} \frac{b_u}{\sqrt{w}}e^{-\lambda w} \cdot dw. \nonumber
\end{gather}
Since $w \in [b_u^2, (b_u+\varepsilon)^2]$ we have $\frac{1}{\sqrt{w}} \in \Bigl[ \frac{1}{b_u+\varepsilon}, \frac{1}{b_u}\Bigr]$ and $\frac{b_u}{\sqrt{w}} \in \Bigl[ 1 - \frac{\varepsilon}{b_u+\varepsilon}, 1\Bigr]$. Thus,
\begin{gather}
F(u)=\frac{(1-\tilde{\theta}_3(u))}{2b_u}\int_{b_u^2}^{(b_u+\varepsilon)^2} e^{-\lambda w} \cdot dw
\nonumber
\end{gather}
where $0 \le \tilde{\theta}_3(u) \le \max_{u \in [0, \varepsilon]}\frac{\varepsilon}{b_u+\varepsilon} \le \frac{\varepsilon}{\rho - c -c\varepsilon+\varepsilon} =\frac{\varepsilon}{\rho - c +\varepsilon (1-c)} \le  \frac{\varepsilon}{\rho - c}$.
Hence,
\begin{gather}
F(u)=\frac{(1-\tilde{\theta}_3(u))}{2b_u} \cdot \frac{e^{-\lambda b_u^2} - e^{-\lambda (b_u+\varepsilon)^2}}{\lambda}  =\frac{(1-\tilde{\theta}_3(u))}{2b_u} \cdot \frac{e^{-\lambda b_u^2}}{\lambda}\cdot(1-\tilde{\theta}_{4}(u))\label{5STATI14final}
\end{gather}
where
\begin{gather}
0 \le \tilde{\theta}_{4}(u) \le \max_{u \in [0, \varepsilon]}\frac{e^{-\lambda (b_u+\varepsilon)^2}}{e^{-\lambda b_u^2}} = \max_{u \in [0, \varepsilon]} e^{-\lambda(2(\rho-c-cu) \varepsilon + \varepsilon^2)} \le e^{-\lambda(2\varepsilon(\rho - c - c \varepsilon)+ \varepsilon^2)} = e^{-\lambda\varepsilon(2(\rho - c - c \varepsilon)+ \varepsilon)} \nonumber.
\end{gather}
Combining  (\ref{5STATI12I14}) and (\ref{5STATI14final}) we see that
\begin{gather}
I_{3} =(1+\theta_{2})\cdot \frac{e^{-\lambda(\rho^2-2c\rho+1)}}{\rho} \cdot  \int_{0}^{\varepsilon}  du \cdot
 e^{- \lambda(u^2 +2u(1-c\rho) - (\rho-c-cu)^2)} \cdot \frac{(1-\tilde{\theta}_3(u))}{2b_u} \cdot \frac{e^{-\lambda b_u^2}}{\lambda}\cdot(1-\tilde{\theta}_{4}(u)) = \nonumber \\
= e^{-\lambda(\rho^2-2c\rho+1)} \cdot \frac{(1+\theta_{2})(1-\theta_3)(1-\theta_{4})}{2\lambda \rho (\rho-c)} \cdot \int_{0}^{\varepsilon}  du \cdot \frac{\rho-c}{b_u} \cdot  e^{- \lambda(u^2 +2u(1-c\rho))} \nonumber
\end{gather}
where $0 \le \theta_3 \le \frac{\varepsilon}{\rho - c}$, $0 \le \theta_4 \le e^{-\lambda \varepsilon (2(\rho - c- c \varepsilon)+\varepsilon)}$. If $0 \le u \le \varepsilon$ then $b_u \in [\rho-c-c\varepsilon, \rho -c]$, whence
$\frac{\rho-c}{b_u} \in [1, 1+ \frac{c\varepsilon}{\rho-c-c\varepsilon}]$. Hence,
\begin{gather}
I_{3} = e^{-\lambda(\rho^2-2c\rho+1)} \cdot \frac{(1+\theta_{2})(1-\theta_3)(1-\theta_{4})(1+\theta_{5})}{2\lambda \rho (\rho-c)} \cdot I_{5} \label{5STATI15I12}
\end{gather}
where
\begin{gather}
0 \le \theta_{5} \le \frac{c\varepsilon}{\rho-c-c\varepsilon}, \ \ \
I_{5}:=\int_{0}^{\varepsilon}  du \cdot  e^{- \lambda(u^2 +2u(1-c\rho))}.\nonumber
\end{gather}
By definition, $\gamma:=1-c\rho$. Then $\gamma > 0$ and
\begin{gather}
I_{5}=\int_{0}^{\varepsilon}  du \cdot  e^{- \lambda(u^2 +2\gamma u)} = e^{\lambda \gamma^2} \cdot \int_{0}^{\varepsilon} e^{- \lambda(u +\gamma)^2}du = e^{\lambda \gamma^2} \cdot \int_{\gamma}^{\gamma+\varepsilon} e^{- \lambda u^2}du = e^{\lambda \gamma^2} \cdot \int_{\gamma^2}^{(\gamma+\varepsilon)^2} e^{- \lambda t}d\sqrt{t} =  \nonumber \\
= \frac{e^{\lambda \gamma^2}}2 \cdot \int_{\gamma^2}^{(\gamma+\varepsilon)^2} \frac1{\sqrt{t}} e^{- \lambda t}dt= \frac{e^{\lambda \gamma^2}}{2\gamma} \cdot \int_{\gamma^2}^{(\gamma+\varepsilon)^2} \frac{\gamma}{\sqrt{t}} e^{- \lambda t}dt. \nonumber
\end{gather}
If  $\gamma^2 \le t \le (\gamma+\varepsilon)^2$ then $1 - \frac{\varepsilon}{\gamma+\varepsilon} =\frac{\gamma}{\sqrt{(\gamma+\varepsilon)^2}}\le \frac{\gamma}{\sqrt{t}}\le \frac{\gamma}{\sqrt{\gamma^2}}=1$ so
\begin{gather}
I_{5}= \frac{(1-\theta_{6})e^{\lambda \gamma^2}}{2\gamma} \cdot \int_{\gamma^2}^{(\gamma+\varepsilon)^2} e^{- \lambda t}dt \nonumber
\end{gather}
where $0 \le \theta_{6} \le \frac{\varepsilon}{\gamma+\varepsilon} = \frac{\varepsilon}{1-c\rho+\varepsilon}$. Thus,
\begin{gather}
I_{5}= \frac{(1-\theta_{6})e^{\lambda \gamma^2}}{2\gamma} \cdot
\frac{e^{- \lambda \gamma^2}-e^{- \lambda (\gamma+\varepsilon)^2}}{\lambda}= \frac{(1-\theta_{6})}{2\gamma\lambda} \cdot (1-\theta_{7}) \nonumber
\end{gather}
where $\theta_{7}  = e^{- \lambda (\gamma+\varepsilon)^2} \cdot e^{\lambda \gamma^2}=e^{-\lambda \varepsilon(2\gamma + \varepsilon)}$.
Hence,
\begin{gather}
I_{5}= \frac{(1-\theta_{6})(1-\theta_{7})}{2\lambda(1-c\rho)},  \ \ \ \theta_{7} =e^{-\lambda \varepsilon(2(1-c\rho) + \varepsilon)}  \nonumber
\end{gather}
whence, taking (\ref{5STATI15I12}) into account, we find that
\begin{gather}
I_{3} = e^{-\lambda(\rho^2-2c\rho+1)} \cdot \frac{(1+\theta_{2})(1-\theta_3)(1-\theta_{4})(1+\theta_{5})}{2\lambda \rho (\rho-c)} \cdot \frac{(1-\theta_{6})(1-\theta_{7})}{2\lambda(1-c\rho)} = \nonumber \\
=e^{-\lambda(\rho^2-2c\rho+1)} \cdot \frac{(1+\theta_{2})(1-\theta_3)(1-\theta_{4})(1+\theta_{5})(1-\theta_{6})(1-\theta_{7})}{4\lambda^2 \rho (\rho-c)(1-c\rho)}. \label{5STATI19}
\end{gather}
Now let us get a convenient expression for $I_{4}$. We need the following notation: ${S}(t_1, t_2):=t_1^2 - 2 c {\rho} t_1 t_2 + {\rho}^2t_2^2$.
Note that (for ${\lambda}>1$)
\begin{gather}
0 \le I_{4}= \int_{{D}_{\varepsilon}} dt_1 dt_2 e^{- ({\lambda}-1){S}(t_1,t_2)} \cdot (t_1 t_2)^{\delta + \frac12} \cdot e^{- (t_1^2- 2c{\rho} t_1 t_2 + {\rho}^2 t_2^2)  } \le \nonumber\\
\le e^{- ({\lambda}-1)\min_{{D}_{\varepsilon}}{S}(t_1,t_2)} \cdot I_{6} \label{5STATi4i12}
\end{gather}
where $I_{6}=\int_{{D}_{\varepsilon}} dt_1 dt_2 \cdot (t_1 t_2)^{\delta + \frac12} \cdot e^{- (t_1^2- 2c {\rho} t_1 t_2 + {\rho}^2 t_2^2)  }$.
Recall that \\ ${D}_{\varepsilon}:=\{(t_1, t_2): t_1 \ge 1, t_2 \ge 1, \max(t_1-1), \rho (t_2 - 1) \ge \varepsilon \}$. We assume that $\varepsilon < \frac1{c}-\rho $ and consequently  $\frac1{c\rho}>1+\frac{\varepsilon}{\rho}$. It is easy to see that $\min_{{D}_{\varepsilon}} S(t_1, t_1) = \min(A_1, A_2, A_3)$ where
\begin{equation}
A_1:=\min_{t_2 \ge \frac{1}{c\rho}} \min_{t_1 \ge 1} S(t_1, t_2), \ \ \
A_2:=\min_{1 + \frac{\varepsilon}{\rho} \le t_2 \le \frac{1}{c\rho}} \min_{t_1 \ge 1} S(t_1, t_2), \ \ \
A_3:=\min_{1 \le t_2 \le1 + \frac{\varepsilon}{\rho}} \min_{t_1 \ge 1+\varepsilon} S(t_1, t_2).
\nonumber
\end{equation}
Wherein
\begin{gather}
A_1 = \min_{t_2 \ge \frac{1}{c\rho}}\min_{t_1 \ge 1} [(t_1 - c \rho t_2)^2+\rho^2(1-c^2)t_2^2]=\min_{t_2 \ge \frac{1}{c\rho}} [0^2+\rho^2(1-c^2)t_2^2]= \rho^2(1-c^2)\frac{1}{(c\rho)^2}=\frac{1-c^2}{c^2}. \nonumber
\end{gather}
Further,
\begin{gather}
A_2=\min_{1 + \frac{\varepsilon}{\rho} \le t_2 \le \frac{1}{c\rho}} \min_{t_1 \ge 1} [(t_1 - c \rho t_2)^2+\rho^2(1-c^2)t_2^2] = \min_{1 + \frac{\varepsilon}{\rho} \le t_2 \le \frac{1}{c\rho}} [(1 - c \rho t_2)^2+\rho^2(1-c^2)t_2^2] = \nonumber \\
=\min_{1 + \frac{\varepsilon}{\rho} \le t_2 \le \frac{1}{c\rho}} (\rho^2 t_2^2 - 2c\rho t_2 + 1) =
1-c^2 + \rho^2 \cdot \min_{1 + \frac{\varepsilon}{\rho} \le t_2 \le \frac{1}{c\rho}} \(t_2 - \frac{c}{\rho}\)^2. \nonumber
\end{gather}
Note that $\frac{c}{\rho} < 1$ since $c < \rho$. Therefore,
\begin{gather}
A_2= 1-c^2 + \rho^2 \cdot \(1 +  \frac{\varepsilon}{\rho} - \frac{c}{\rho}\)^2= 1-c^2 + (\rho-c+\varepsilon)^2=1-c^2+\rho^2+c^2+\varepsilon^2 - 2c\rho+2\rho\varepsilon -2c\varepsilon = \nonumber \\ =(\rho^2-2c\rho+1)+\varepsilon(2(\rho-c)+\varepsilon)=S(1,1)+\varepsilon(2(\rho-c)+\varepsilon).
\nonumber
\end{gather}
Further, we assume that the inequality $\varepsilon < \frac{1}c - \rho$ holds. Therefore, $c\rho t_2 \le c(\rho+\varepsilon) < 1$ for $t_2 \le  1 + \frac{\varepsilon}{\rho}$. Hence,
\begin{gather}
A_3=\min_{1 \le t_2 \le 1 + \frac{\varepsilon}{\rho}} \min_{t_1 \ge 1+\varepsilon} [(t_1 - c \rho t_2)^2+\rho^2(1-c^2)t_2^2] =\min_{1 \le t_2 \le 1 + \frac{\varepsilon}{\rho}}  [(1+\varepsilon - c \rho t_2)^2+\rho^2(1-c^2)t_2^2] = \nonumber \\
=\min_{1 \le t_2 \le 1 + \frac{\varepsilon}{\rho}}  (c^2 \rho^2 t_2^2 - 2c(1+\varepsilon)\rho t_2 + (1+\varepsilon)^2 + \rho^2(1-c^2)t_2^2) = \nonumber \\
=(1+\varepsilon)^2 - c^2 (1+\varepsilon)^2 + \min_{1 \le t_2 \le 1 + \frac{\varepsilon}{\rho}}(\rho t_2 - c(1+\varepsilon))^2 = (1-c^2) (1+\varepsilon)^2 + \rho^2 \cdot \min_{1 \le t_2 \le 1 + \frac{\varepsilon}{\rho}} \Bigl(t_2 - \frac{c(1+\varepsilon)}{\rho}\Bigr)^2. \nonumber
\end{gather}
We assume that  $\rho-c-c\varepsilon>0$ (i.e. $\varepsilon < \frac{\rho}{c}-1$). Therefore, $\frac{c(1+\varepsilon)}{\rho} < 1$ and thus
\begin{gather}
A_3 = (1-c^2) (1+\varepsilon)^2 + \rho^2 \(1 - \frac{c(1+\varepsilon)}{\rho}\)^2 = (1-c^2) (1+\varepsilon)^2  + \rho^2 - 2c \rho (1 + \varepsilon) + c^2 (1 + \varepsilon)^2 = \nonumber \\
=\rho^2 - 2\rho c + 1 + \varepsilon(2(1-c\rho)+\varepsilon) = S(1,1) + \varepsilon(2(1-c\rho)+\varepsilon) \nonumber
\end{gather}
and $2(1-c\rho)+\varepsilon>0$ since $\rho < \frac1{c}$.
Hence,
\begin{gather}
\min_{{D}_{\varepsilon}} S(t_1, t_2)= \min \Bigl( \frac{1-c^2}{c^2}, S(1,1)+\varepsilon(2(\rho-c)+\varepsilon), S(1,1) + \varepsilon(2(1-c\rho)+\varepsilon)\Bigr)= \nonumber \\
=\min \Bigl( \frac{1-c^2}{c^2}, S(1,1)+\varepsilon(2\min(\rho-c, 1-c\rho)+\varepsilon) \Bigr). \nonumber
\end{gather}
Let us show that $\frac{1-c^2}{c^2} > \rho^2 - 2c \rho+1$\label{5STAT1-c^2versus}. In fact, $c < \rho < \frac1{c}$ therefore $\frac1{c}-c > \rho - c > 0$, whence $(\frac{1-c^2}{c})^2>(\rho-c)^2$. Hence, $(1-c^2)\frac{1-c^2}{c^2}+(1-c^2) > \rho^2 - 2c\rho+c^2+(1-c^2)$, thus $\frac{1-c^2}{c^2} > \rho^2 - 2c\rho+1=S(1,1)$. We require  $\varepsilon$ be small enough that inequality $\varepsilon(2\min(\rho-c, 1-c\rho)+\varepsilon) <  \frac{1-c^2}{c^2} - S(1,1)$ holds true. Then
\begin{gather}
\min_{{D}_{\varepsilon}} S(t_1, t_2)=  S(1,1)+\varepsilon(2\min(\rho-c, 1-c\rho)+\varepsilon) = \nonumber \\ = \rho^2 - 2\rho c + 1 +\varepsilon(2\min(\rho-c, 1-c\rho)+\varepsilon) \nonumber
\end{gather}
whence by virtue of (\ref{5STATi4i12}) we obtain that
\begin{gather}
0 \le I_{4} \le e^{- ({\lambda}-1)({\rho}^2 - 2c {\rho} +1 + \varepsilon(2\min({\rho}-c, 1-c{\rho})+\varepsilon))} \cdot I_{6} \label{5STATI20I21}
\end{gather}
where $I_{6}=\int_{{D}_{\varepsilon}} dt_1 dt_2 \cdot (t_1 t_2)^{\delta + \frac12} \cdot e^{- (t_1^2- 2c {\rho} t_1 t_2 + {\rho}^2 t_2^2)  }$. Since the integral $I_{6}$ converges,  $I_{4}$ (for fixed $\varepsilon$) gives an exponentially small  (in $\lambda$) contribution to $I_{2}$ as compared to $I_{3}$. We estimate this contribution, taking into account that due to the exponential smallness the upper estimate for the integral  $I_{6}$  can be made quite coarse. It is easy to see that
\begin{gather}
0 \le I_{6}=\int_{{D}_{\varepsilon}} dt_1 dt_2 \cdot (t_1 t_2)^{\delta + \frac12} \cdot e^{- (t_1^2- 2c {\rho} t_1 t_2 + {\rho}^2 t_2^2)  } \le
\int_{1}^{\infty} dt_1 \int_{1}^{\infty} dt_2 \cdot (t_1 t_2)^{\delta + \frac12} \cdot e^{- (t_1^2- 2c {\rho} t_1 t_2 + {\rho}^2 t_2^2)  } =\nonumber \\
=\frac1{{\rho}^{\delta+\frac32}}\cdot \int_{1}^{\infty} dt_1  \cdot t_1^{\delta + \frac12} \cdot e^{- t_1^2} \int_{1}^{\infty} {\rho} dt_2 \cdot ({\rho}t_2)^{\delta + \frac12} e^{-(\rho^2t_2^2-2c{\rho}t_1 t_2)} =\nonumber\\
=\frac1{{\rho}^{\delta+\frac32}}\cdot \int_{1}^{\infty} dt_1  \cdot t_1^{\delta + \frac12} \cdot e^{- t_1^2} \int_{{\rho}}^{\infty} du \cdot u^{\delta + \frac12} e^{-(u^2-2ct_1 u)}\le \nonumber \\
\le \frac1{{\rho}^{\delta+\frac32}}\cdot \int_{0}^{\infty} dt_1  \cdot t_1^{\delta + \frac12} \cdot e^{- t_1^2} \int_{{\rho}}^{\infty} du \cdot u^{\delta + \frac12} e^{-(u^2-2ct_1 u)}= \nonumber
\\
= \frac1{{\rho}^{\delta+\frac32}}\cdot \int_{0}^{\infty} dt_1  \cdot t_1^{\delta + \frac12} \cdot e^{- (1-c^2)t_1^2} \int_{{\rho}}^{\infty} du \cdot u^{\delta + \frac12} e^{-(u-ct_1)^2} = \nonumber \\
= \frac1{{\rho}^{\delta+\frac32}}\cdot \int_{0}^{\infty} dt_1  \cdot t_1^{\delta + \frac12} \cdot e^{- (1-c^2)t_1^2} \cdot \int_{{\rho} - ct_1}^{\infty} (v+ct_1)^{\delta + \frac12} e^{-v^2} \cdot dv = \nonumber \\
= \frac1{{\rho}^{\delta+\frac32}}\cdot \int_{0}^{\infty} dt_1  \cdot t_1^{\delta + \frac12} \cdot e^{- (1-c^2)t_1^2} \cdot I_{7}(t_1) \label{5STATI21I22}
\end{gather}
where $I_{7}(t_1):=\int_{{\rho}-ct_1}^{\infty} (v+ct_1)^{\delta + \frac12} e^{-v^2} \cdot dv$. Recall that  $\delta=\frac{N-3}2$ where $N\ge 2$ (the case $N=1$ is trivial and we don't consider it), so $\delta+\frac{1}2 \ge 0$.  Further,
\begin{gather}
0 \le I_{7}(t_1) = \int_{{\rho}-ct_1}^{\infty}  (v+ct_1)^{\delta + \frac12} e^{-v^2} \cdot dv \le \int_{-ct_1}^{\infty}  (v+ct_1)^{\delta + \frac12} e^{-v^2} \cdot dv =  \nonumber \\
= \int_{-ct_1}^{ct_1}  (v+ct_1)^{\delta + \frac12} e^{-v^2} \cdot dv + \int_{ct_1}^{\infty}  (v+ct_1)^{\delta + \frac12} e^{-v^2} \cdot dv \le \nonumber \\
\le (2ct_1)^{\delta + \frac12} \int_{-ct_1}^{ct_1}  e^{-v^2} \cdot dv + \int_{ct_1}^{\infty}  (v+ct_1)^{\delta + \frac12} e^{-v^2} \cdot dv.\nonumber
\end{gather}
Since $(v+ct_1)^{\delta + \frac12} \le (2v)^{\delta + \frac12}$ for $v \ge ct_1 \ge 0$, we have
\begin{gather}
0 \le I_{7}(t_1) \le (2ct_1)^{\delta + \frac12} \int_{-ct_1}^{ct_1}  e^{-v^2} \cdot dv + \int_{ct_1}^{\infty}  (v+ct_1)^{\delta + \frac12} e^{-v^2} \cdot dv \le \nonumber \\
\le (2ct_1)^{\delta + \frac12} \int_{-ct_1}^{ct_1}  e^{-v^2} \cdot dv + \int_{ct_1}^{\infty}  (2v)^{\delta + \frac12} e^{-v^2} \cdot dv \le \nonumber \\
\le (2ct_1)^{\delta + \frac12} \int_{\mathbb{R}}  e^{-v^2} \cdot dv + \int_{0}^{\infty}  (2v)^{\delta + \frac12} e^{-v^2} \cdot dv = \nonumber \\
= (2ct_1)^{\delta + \frac12} \sqrt{\pi}+2^{\delta + \frac12} \cdot \int_{0}^{\infty}  u^{\frac{\delta}2 + \frac14} e^{-u} \cdot \frac{du}{2\sqrt{u}}= \nonumber
(2ct_1)^{\delta + \frac12} \sqrt{\pi}+2^{\delta - \frac12} \cdot \int_{0}^{\infty}  u^{\frac{\delta}2 + \frac34-1} e^{-u} \cdot du= \nonumber  \\=
 (2c)^{\delta + \frac12} \sqrt{\pi}  t_1^{\delta + \frac12} +2^{\delta - \frac12} \cdot \Gamma\(\frac{\delta}2 + \frac34\) = B_1 \cdot  t_1^{\delta + \frac12} + B_2\nonumber
\end{gather}
where $B_1  = (2c)^{\delta + \frac12} \sqrt{\pi} > 0$ and  $B_2 =2^{\delta - \frac12} \cdot \Gamma\(\frac{\delta}2 + \frac34\)> 0$. Hence, by virtue of (\ref{5STATI21I22}) we obtain the following relations:
\begin{gather}
0\le {\rho}^{\delta+\frac32}\cdot I_{6} \le \int_{0}^{\infty} dt_1  \cdot t_1^{\delta + \frac12} \cdot e^{- (1-c^2)t_1^2} \cdot I_{7}(t_1) \le \int_{0}^{\infty} dt_1  \cdot t_1^{\delta + \frac12} \cdot e^{- (1-c^2)t_1^2} \cdot \( B_1 \cdot  t_1^{\delta + \frac12} + B_2\) = \nonumber \\
=B_1 \cdot \int_{0}^{\infty} dt_1  \cdot t_1^{\tilde{\delta} + \frac12} \cdot e^{- (1-c^2)t_1^2} + B_2 \cdot
\int_{0}^{\infty} dt_1  \cdot t_1^{\delta + \frac12} \cdot e^{- (1-c^2)t_1^2} \label{5STATdlyaZ}
\end{gather}
where $\tilde{\delta}:=2\delta+\frac12$. Let $$Z(\delta):=\int_{0 }^{\infty} u^{\delta + \frac12} \cdot e^{-(1-c^2)u^2} du.$$
We make the substitution $t=(1-c^2)u^2$ (whence $u=\frac{\sqrt{t}}{\sqrt{1-c^2}}$). We get:
\begin{gather}
Z(\delta) = \int_{0 }^{\infty} \Biggl(\sqrt{\frac{t}{1-c^2}}\Biggr)^{\delta + \frac12} \cdot e^{-t} \cdot \frac{dt}{2\sqrt{t}\sqrt{1-c^2}} = \frac{1}{2(1-c^2)^{\frac{\delta}2+\frac{3}4 }} \int_{0 }^{\infty} t^{\frac{\delta}2+\frac34-1 } e^{-t} dt =\frac{\Gamma(\frac{\delta}2+\frac34)}{2(1-c^2)^{\frac{\delta}2+\frac{3}4 }}. \nonumber
\end{gather}
From (\ref{5STATdlyaZ}) if follows that
\begin{gather}
0 \le {\rho}^{\delta+\frac32} \cdot I_{6}  \le
B_1 \cdot \frac{\Gamma(\frac{\tilde{\delta}}2+\frac34)}{2(1-c^2)^{\frac{\tilde{\delta}}2+\frac{3}4 }} + B_2 \cdot \frac{\Gamma(\frac{\delta}2+\frac34)}{2(1-c^2)^{\frac{\delta}2+\frac{3}4 }} = \nonumber \\
=B_1 \cdot \frac{\Gamma(\delta+1)}{2(1-c^2)^{\delta+1}} + B_2 \cdot \frac{\Gamma(\frac{\delta}2+\frac34)}{2(1-c^2)^{\frac{\delta}2+\frac{3}4 }}. \nonumber
\end{gather}

Consequently,
\begin{gather}
0 \le I_{6}  \le \frac1{{\rho}^{\delta+\frac32} } \cdot \((2c)^{\delta + \frac12} \sqrt{\pi} \cdot \frac{\Gamma(\delta+1)}{2(1-c^2)^{\delta+1}} + 2^{\delta - \frac12} \cdot \Gamma\(\frac{\delta}2 + \frac34\) \cdot \frac{\Gamma(\frac{\delta}2+\frac34)}{2(1-c^2)^{\frac{\delta}2+\frac{3}4 }}
 \)  \nonumber.
\end{gather}
Combining the relations (\ref{5STATI18I19I20}), (\ref{5STATI19}) and (\ref{5STATI20I21}) we see that
\begin{gather}
I_{2} = I_{3}+I_{4} = e^{-{\lambda}({\rho}^2-2c{\rho}+1)} \cdot \frac{(1+\theta_{2})(1-\theta_{3})(1-\theta_{4})(1+\theta_{5})(1-\theta_{6})(1-\theta_{7})}{4{\lambda}^2 {\rho} ({\rho}-c)(1-c{\rho})} + I_{4} \label{5STATI18final}
\end{gather}
and
\begin{gather}
0 \le I_{4} \le e^{- ({\lambda}-1)({\rho}^2 - 2c {\rho} +1 + \varepsilon(2\min({\rho}-c, 1-c{\rho})+\varepsilon))} \cdot I_{6}. \nonumber
\end{gather}
Now we obtain a formula for $\alpha$,  taking into account the relations (\ref{5STATalphaI1}), (\ref{5STATI17I18}) and (\ref{5STATI18final}). Namely,
\begin{gather}
\alpha = \frac{2C}{\sqrt{\pi c}} (1 + \theta_1) I_{1} = \frac{2c}{(K_2 c)^{1+\delta}\Gamma(1+\delta)\sqrt{\pi c}} \cdot (1 + \theta_1) ({x}_1 {x}_2)^{(\delta+\frac{3}2 )} \cdot I_{2} \label{5STAT1fladlyath20}
\end{gather}
where
\begin{gather}
I_{2} = e^{-{\lambda}({\rho}^2-2c{\rho}+1)} \cdot \(\frac{(1+\theta_{2})(1-\theta_{3})(1-\theta_{4})(1+\theta_{5})(1-\theta_{6})(1-\theta_{7})}{4{\lambda}^2 {\rho} ({\rho}-c)(1-c{\rho})} + \tilde{I}_{4}\) \label{5STAT2fladlyath20}
\end{gather}
and
\begin{gather}
0 \le \tilde{I}_{4}:=I_{4} \cdot e^{{\lambda}({\rho}^2-2c{\rho}+1)} \le e^{- ({\lambda}-1)({\rho}^2 -2c {\rho} +1 + \varepsilon(2\min({\rho}-c, 1-c{\rho})+\varepsilon))} \cdot e^{{\lambda}({\rho}^2-2c{\rho}+1)} \cdot I_{6} = \nonumber \\
= e^{- {\lambda}\varepsilon(2\min({\rho}-c, 1-c{\rho})+\varepsilon)} \cdot
e^{({\rho}^2 - 2c {\rho} +1 + \varepsilon(2\min({\rho}-c, 1-c{\rho})+\varepsilon))}I_{6}. \nonumber
\end{gather}  We replace $\delta$ and $K_2$ by $\frac{N-3}2$ and $\frac1{1-c^2}$ respectively. Recall that ${x}_1 = \sqrt{\frac{x_1^*}{2(1-c^2)}}$ and $ {x}_2 = \sqrt{ \frac{x_2^*}{2 (1-c^2)}}$. Hence,
\begin{gather}
\alpha =  \frac{2c}{\(\frac{c}{1-c^2}\)^{\frac{N-1}2}\Gamma(\frac{N-1}2)\sqrt{c \pi}} \cdot  ({x}_1 {x}_2)^{\frac{N}2} \cdot I_{2} (1+ \theta_1) = \nonumber \\
 =\frac{2c(1-c^2)^{\frac{N-1}2}}{c^{\frac{N-1}2}\Gamma(\frac{N-1}2)\sqrt{c \pi}} \cdot  \(\sqrt{\frac{x_1^*}{2(1-c^2)} \cdot \frac{x_2^*}{2(1-c^2)}}\)^{\frac{N}2} \cdot I_{2} (1+ \theta_1) = \nonumber \\
 =\frac{2(1-c^2)^{\frac{N-1}2}(x_1^* x_2^*)^{\frac{N}4}}{c^{\frac{N-2}2}\Gamma(\frac{N-1}2)\sqrt{ \pi}[2(1-c^2)]^\frac{N}2} \cdot  I_{2} (1+ \theta_1)
 =\frac{(x_1^* x_2^*)^{\frac{N}4}}{2^{\frac{N}{2}-1}c^{\frac{N}2-1}\Gamma(\frac{N-1}2)\sqrt{ \pi(1-c^2)}} \cdot  I_{2} (1+ \theta_1). \nonumber
\end{gather}
Thus, Proposition \ref{5STATpredl1} is proved.

We proceed to the proof of Theorem  \ref{5STATth2}. Since ${\rho}=\frac{{x}_2}{{x}_1}$ we have $\frac{x_2^*}{x_1^*}={\rho}^2$. Hence, for $\lambda:=x_1^2 \to +\infty$ the following equivalence holds true:
\begin{gather}
\alpha \sim  \frac{(x_1^* x_2^*)^{\frac{N}4}}{(2c)^{\frac{N}{2}-1}\Gamma(\frac{N-1}2)\sqrt{ \pi(1-c^2)}}
\cdot e^{-{\lambda}({\rho}^2-2c{\rho}+1)} \cdot \frac{1}{4{\lambda}^2 {\rho} ({\rho}-c)(1-c{\rho})} =\nonumber \\
=
\frac{(x_1^* {\rho}^2 x_1^*)^{\frac{N}4} e^{-{x}_1^2({\rho}^2-2c{\rho}+1)}}{(2c)^{\frac{N}{2}-1}\Gamma(\frac{N-1}2)\sqrt{ \pi(1-c^2)}}
 \cdot \frac{1}{4{x_1}^4 {\rho} ({\rho}-c)(1-c{\rho})} = \nonumber \\
 =\frac{(x_1^* {\rho})^{\frac{N}2} e^{-{x}_1^2({\rho}^2-2c{\rho}+1)}}{(2c)^{\frac{N}{2}-1}\Gamma(\frac{N-1}2)\sqrt{ \pi(1-c^2)}}
 \cdot \frac{(2(1-c^2))^2}{4{x_1^*}^2 {\rho} ({\rho}-c)(1-c{\rho})} = \nonumber \\
 =\frac{(x_1^*)^{\frac{N}2-2} {\(\frac{{\rho}}{2c}\)}^{\frac{N}2-1} \cdot (1-c^2)^{\frac32} \cdot e^{-\frac{x_1^*}{2(1-c^2)}({\rho}^2-2c{\rho}+1)}}{\sqrt{\pi}\Gamma(\frac{N-1}2)({\rho}-c)(1-c{\rho})}. \nonumber
\end{gather}
Let us estimate the error of the approximation obtained (for $\lambda \to +\infty$).  For this we fix
the number $a>0$. We put
$$\varepsilon(\lambda):=\frac{a \ln \lambda}{\lambda}.$$ Then $\varepsilon \to 0$ and (in the notation of Proposition \ref{5STATpredl1}) $\theta_i = O(\varepsilon)$, $i=2,3,5,6$. Further, due to choosing a sufficiently large (fixed) number $a$ we can assume that $\theta_i = O(\frac{1}{\lambda})$, $i=4,7$ and $\tilde{I}_4 = O(\frac{1}{\lambda^3})$. Since $\Psi(\nu, x) = O(\frac{1}{x})$ as $x \to +\infty$ then $\theta_1 = O(\frac{1}{x_1^*})$. Finally,
notice, that
\begin{gather}
O(\varepsilon) = O\( \frac{a\ln \lambda}{\lambda} \) = O \(\frac{\ln x_1^*}{x_1^*}\), \ \ x_1^* \to +\infty \nonumber
\end{gather}
since $\lambda = \frac{x_1^*}{2\beta}$. Theorem \ref{5STATth2} is proved.
\begin{remark}
Note that in the case  $N=2$ (and consequently $\delta < 0$)  all results are preserved with natural changes due to Remark \ref{5STATzame4oBessel} and small changes will concern only that part of the reasoning in which Infeld's function is replaced by its asymptotics.\end{remark}
\begin{remark}
If $\frac{x_2}{x_1} < c$ then  it is easy to show (using analogous arguments) that $$argmin_{t_1 \ge 1, t_2 \ge 1}(t_1^2- 2c \rho t_1 t_2 + \rho^2 t_2^2) = \(1, \frac{c}{\rho}\),$$
and then use arguments similar to the proof of Theorem  \ref{5STATth2} and obtain a formula for $\alpha$ in this case.  Similarly, we may consider the case  $\frac{x_2}{x_1} > \frac1{c}$. We also note that both of these cases are less interesting from a practical point of view.
\end{remark}
\section{Proof of Theorem \ref{5STATth3}}
\begin{lemma} \label{5STATlemma1}
Let $k\ge 2$ be fixed. Then the following relation holds:
\begin{gather}
\int_{x}^{\infty}\frac{t^{\frac{k}2-1}\cdot e^{-\frac{t}2}}{2^{\frac{k}2 }}dt=\(\frac{x}2\)^{\frac{k}2-1}\cdot e^{-\frac{x}2} \cdot \(1+{o}(1)  \), \ \ \ x \to +\infty. \nonumber
\end{gather}
\end{lemma}
\begin{proof}
By the L'Hospital's rule
\begin{gather}
\lim_{x \to +\infty} \frac{\int_{x}^{\infty}\frac{t^{\frac{k}2-1}\cdot e^{-\frac{t}2}}{2^{\frac{k}2 }}dt}{\(\frac{x}2\)^{\frac{k}2-1}\cdot e^{-\frac{x}2}}=\lim_{x \to +\infty}  \frac{   - \( \frac{x^{\frac{k}2-1}\cdot e^{-\frac{x}2}}{2^{\frac{k}2 }}  \) }{   \(\frac{x}2\)^{\frac{k}2-1}\cdot e^{-\frac{x}2} \cdot \( - \frac12\) + (\frac{k}2-1)\cdot \(\frac{x}2\)^{\frac{k}2-2} \cdot \frac 12 \cdot e^{-\frac{x}2}   } = 1. \nonumber
\end{gather}
This proves the lemma.
\end{proof}
Now let us prove Theorem \ref{5STATth3}. Recall that   $\alpha_1:=\lim_{n_1 \to \infty}P(X(n_1)>x_1^*)$, \\ $\alpha_2:=\lim_{n_2 \to \infty}P(X(n_2)>x_2^*)$.
Further, within the framework of our model the distribution $\chi^2(N-1)$ is the limit distribution for both  $X(n_1)$ and $X(n_2)$. It is known that  $x_1^* \to +\infty$ and ${x}_2^* \to + \infty$ (and therefore  $\alpha_1 \to +0$ and $\alpha_2 \to +0$ ) so that $\sqrt{\frac{\ln \alpha_2}{\ln \alpha_1}}=P$. Hence, $\alpha_2$ can be expressed via $\alpha_1$, and the numbers  $x_1^*$ and $x_2^*$ (and therefore also the numbers  $x_1 = \sqrt{\frac{x_1^*}{2\beta}}$, $x_2 = \sqrt{ \frac{x_2^*}{2 \beta}}$) are determined by the numbers  $\alpha_1$, $\alpha_2$ uniquely. Further, it follows from Lemma \ref{5STATlemma1} that \begin{gather}
\alpha_1\Gamma\(\frac{N-1}2\) = \(\frac{x_1^*}2\)^{\frac{N-3}2} \cdot e^{-\frac{x_1^* }{2} }  \cdot
\( 1 + {o}(1) \),  \ \ \ \
\alpha_2\Gamma\(\frac{N-1}2\) = \(\frac{x_2^*}2\)^{\frac{N-3}2} \cdot e^{-\frac{x_2^* }{2} }  \cdot
\( 1 + {o}(1) \),\nonumber
\end{gather} therefore
$x_i^* \sim -2 \ln \alpha_i$, $i = 1,2$, where the notation $f(x) \sim g(x)$ means that  $\frac{f(x)}{g(x)} \to 1$ (as $x \to +\infty$).
It is easy to see that $
\alpha_i\Gamma(\frac{N-1}2) \sim \(-\ln \alpha_i\)^{\frac{N-3}2} \cdot e^{-\frac{x_i^* }{2} },
$ therefore $\(-\ln \alpha_i\)^{\frac{3-N}2} \cdot \alpha_i\Gamma(\frac{N-1}2) \sim e^{-\frac{x_i^* }{2} }$
whence
\begin{gather}
\(-\ln \alpha_i\)^{N-3} \cdot \(\alpha_i\Gamma\(\frac{N-1}2\)\)^{-2} \sim e^{x_i^*}. \label{5STATexpx}
\end{gather}
Further, the ratio ${\rho}$ was fixed in Theorem \ref{5STATth2} but now it is not so. Nevertheless, ${\rho}=\frac{{x}_2}{{x}_1} \to P$. Indeed,  ${x}_i: = \sqrt{\frac{x_i^*}{2\beta}} \sim\sqrt{\frac{-2 \ln \alpha_i}{2\beta}} = \sqrt{\frac{-\ln \alpha_i}{1-c^2}}$, it means that ${\rho} = \frac{{x_2}}{{x_1}} \sim \sqrt{\frac{\ln \alpha_2}{\ln \alpha_1}}=P$.
We assume that $c < P < \frac1{c}$. Since ${\rho} \to P$ for sufficiently small  $\alpha_1$ the following inequality holds: $c < \frac{{x}_2}{{x}_1} < \frac1{c}$. Hence, taking to account formulas (\ref{5STAT1fladlyath20}) and (\ref{5STAT2fladlyath20}), obtained in the proof of Proposition \ref{5STATpredl1}, we have:
\begin{gather}
\alpha \sim \frac{c({\rho}{x_1}^2)^{(\delta+\frac{3}2 )} \cdot e^{-{x}_1^2({\rho}^2-2c{\rho}+1)}}{2(K_2 c)^{1+\delta}\Gamma(1+\delta)\sqrt{\pi c}{x}_1^4 {\rho} ({\rho}-c)(1-c{\rho})}. \nonumber
\end{gather}
Since  ${x}_i \sim  \sqrt{\frac{-\ln \alpha_i}{1-c^2}}$ and ${\rho} \sim P$ we obtain
\begin{gather}
\alpha \sim \frac{c \cdot {(P  \cdot \frac{-\ln \alpha_1}{1-c^2}  )}^{(\delta+\frac{3}2 )} \cdot e^{-({x}_2^2 -2c {x}_1 \cdot {x}_2 + {x}_1^2 )}   }{2(K_2 c)^{1+\delta}\Gamma(1+\delta)\sqrt{\pi c}{\( \frac{-\ln \alpha_1}{1-c^2} \)}^{2} P (P-c)(1-c P)} = \nonumber  \\
=\frac{c \cdot {(P  \cdot \frac{-\ln \alpha_1}{1-c^2}  )}^{(\delta+\frac{3}2 )} \cdot \(e^{x_1^*} \cdot e^{x_2^*}  \cdot e^{-2c\sqrt{x_1^* x_2^*}} \)^{-\frac1{2(1-c^2)}} }{2(K_2 c)^{1+\delta}\Gamma(1+\delta)\sqrt{\pi c}{\( \frac{-\ln \alpha_1}{1-c^2} \)}^{2} P (P-c)(1-c P)} . \label{5STATalph}
\end{gather}
It remains to express $e^{\sqrt{x_1^* x_2^*}}$ via $\alpha_i$. We need the following Lemma.
\begin{lemma} \label{5STATlemma2}
Suppose $n \ge 0$ and for $t_i \to +\infty$ the following relations hold: $\alpha_i(t_i) = t_i^n e^{-t_i}\( 1 +{o}(1)\)$, $i =1,2$. Suppose the continious mapping $\alpha_i(t_i)$ is a bijection between $(0, \alpha_i^0]$ and $[t_i^0, +\infty)$.
If $\sqrt{\frac{\ln \alpha_2}{\ln \alpha_1}}=P=const$ then $e^{\sqrt{t_1 t_2}} \sim  \alpha_1^{-P} \cdot \((- \ln \alpha_1)^{P} \cdot (-\ln \alpha_2)^{\frac{1}{P}}\)^{\frac{n}2}$ as $\max(\alpha_1, \alpha_2) \to +0$.
\end{lemma}
\begin{proof}
Note that $\max(\alpha_1, \alpha_2) \to +0$  if and only if  $\min(t_1, t_2) \to +\infty$. Further,
$$\ln \alpha_i = -t_i + n \ln t_i + \ln \(1+ {o}\(1\) \)= -t_i + n \ln t_i + {o}\(1\)$$
whence $t_i \sim -\ln \alpha_i$ and
\begin{gather}t_i = - \ln \alpha_i +  n \ln t_i +  {o}\(1\) = \(- \ln \alpha_i\) \cdot \( 1 + \frac{n \ln t_i +{o}\(1\)}{-\ln \alpha_i} \). \nonumber
\end{gather}
Hence,
\begin{equation}
t_i = \(- \ln \alpha_i\) \cdot (1 + \delta_i(t_i)) \label{5STATt-delta}
\end{equation}
where $\delta_i(t_i):=\frac{n \ln t_i +{o}\(1\)}{- \ln \alpha_i}$. Since $t_i \sim -\ln \alpha_i$ we have (taking (\ref{5STATt-delta}) into account) $ \delta_i(t_i) = {o}(1)$. Thus,
\begin{gather}
\sqrt{t_1 t_2} = \sqrt{(-\ln \alpha_1)(-\ln \alpha_2)(1+\delta_1)(1+\delta_2)}
=\nonumber \\
=\sqrt{\frac{-\ln \alpha_2}{-\ln \alpha_1}}\cdot(-\ln \alpha_1) \cdot [1+\frac{\delta_1}2 (1+{O}(\delta_1))]
\cdot [1+\frac{\delta_2}2 (1+{O}(\delta_2)) ]
= \nonumber \\
=P\cdot(-\ln \alpha_1) \cdot [1+\frac{\delta_1}2 (1+{O}(\delta_1)) + \frac{\delta_2}2 (1+{O}(\delta_2)) + \frac14\delta_1 \delta_2 (1+ {o}(1)) ]. \label{5STATsqrtstep}
\end{gather}
Let us simplify the expression obtained. For this we note that
\begin{gather}
\frac{\delta_2}{\delta_1} \sim \frac{n\ln t_2}{-\ln \alpha_2} \cdot \frac{-\ln \alpha_1}{n\ln t_1} = P^{-2}\frac{\ln t_2}{\ln t_1}. \nonumber
\end{gather}
In this case (by virtue of (\ref{5STATt-delta})) the following relation holds: $\ln t_i = \ln (-(1+\delta_i) \ln \alpha_i)$, whence $\ln t_i = \ln (- \ln \alpha_i) + \ln(1+\delta_i) \sim \ln (- \ln \alpha_i)$. In addition, $-\ln \alpha_2 = P^2 (-\ln \alpha_1)$. Hence,
\begin{gather}
\frac{\delta_2}{\delta_1} \sim  P^{-2}\frac{\ln t_2}{\ln t_1} \sim  P^{-2}\frac{\ln (- \ln \alpha_2)}{\ln (- \ln \alpha_1)} = P^{-2}\frac{\ln (P^2 (-\ln \alpha_1))}{\ln (- \ln \alpha_1)} = P^{-2}\frac{2\ln P + \ln (-\ln \alpha_1)}{\ln (- \ln \alpha_1)} \sim P^{-2} . \nonumber
\end{gather}
Therefore, $\delta_1 \delta_2 (1+ {o}(1))=\frac{\delta_2}{\delta_1} \cdot \delta_1^2 (1+ {o}(1)) = \delta_1^2 {O}(1).$
By (\ref{5STATsqrtstep}) we get
\begin{gather}
\sqrt{t_1 t_2} =
P\cdot(-\ln \alpha_1) \cdot [1+\frac{\delta_1}2 (1+{O}(\delta_1)) + \frac{\delta_2}2 (1+{O}(\delta_1)) + \frac14\delta_1^2 {O}(1)] = \nonumber \\
=P(-\ln \alpha_1)\(1+ \frac{\delta_1}2(1+{O}(\delta_1)) + \frac{\delta_2}2(1+{O}(\delta_2)) \)=
\nonumber \\
=P(-\ln \alpha_1)
 +  \frac{P}2(-\ln \alpha_1){\delta_1}(1+{O}(\delta_1)) +  \frac{P}2(-\ln \alpha_1){\delta_2}(1+{O}(\delta_2))=\nonumber \\
 =P(-\ln \alpha_1)
 +   \frac{P}2(n \ln t_1 + {o}(1))(1+{O}(\delta_1)) +  \frac{P}2 \delta_2 (- \ln \alpha_2)\cdot \(\frac{-\ln \alpha_1}{-\ln \alpha_2} \) (1+{O}(\delta_2))=\nonumber \\
 =P(-\ln \alpha_1) +   \frac{P}2(n \ln t_1 + {o}(1))(1+{O}(\delta_1)) +  \frac{P}2 (n \ln t_2 + {o}(1))P^{-2}(1+{O}(\delta_2)) =\nonumber \\
= P(-\ln \alpha_1) + \frac{Pn}2 \cdot (1+{O}(\delta_1)) \ln t_1 + \frac{n}{2P} \cdot (1+{O}(\delta_2)) \ln t_2 + {o}(1). \nonumber
\end{gather}
Note that $\delta_i \ln t_i \sim \frac{n \ln^2 t_i}{-\ln \alpha_i} \sim \frac{n \ln^2 t_i}{t_i}$. Hence $\delta_i \ln t_i = {o}(1)$. Thus,
\begin{gather}
\sqrt{t_1 t_2} = P(-\ln \alpha_1) + \frac{Pn}2 \cdot \ln t_1 + \frac{n}{2P} \cdot \ln t_2 + {o}(1). \nonumber
\end{gather}
Consequently,
\begin{gather}
e^{\sqrt{t_1 t_2}} \sim \alpha_1^{-P} \cdot t_1^{\frac{Pn}2} \cdot t_2^{\frac{n}{2P}} \sim \alpha_1^{-P} \cdot \((- \ln \alpha_1)^{P} \cdot (-\ln \alpha_2)^{\frac{1}{P}}\)^{\frac{n}2}. \nonumber
\end{gather}
Thus, Lemma is proved.
\end{proof}
Further, $\tilde{\alpha}_i :=\alpha_i\Gamma(\frac{N-1}2) = \(\frac{x_i^*}2\)^{\frac{N-3}2} \cdot e^{-\frac{x_i^* }{2} }  \cdot
\( 1 + {o}(1) \)$. By Lemma \ref{5STATlemma2} and by the fact that $\ln (const \cdot \alpha) \sim  \ln \alpha$ as $\alpha \to +0$, the following relations hold:
\begin{gather}
e^{\sqrt{\frac{x_1^*}2 \cdot \frac{x_2^*}2}} \sim \tilde{\alpha}_1^{-P} \cdot \((- \ln \tilde{\alpha}_1)^{P} \cdot (-\ln \tilde{\alpha}_2)^{\frac{1}{P}}\)^{\frac{N-3}4} \sim \( \Gamma\(\frac{N-1}2\) \alpha_1\)^{-P} \cdot \((- \ln \alpha_1)^{P} \cdot (-\ln \alpha_2)^{\frac{1}{P}}\)^{\frac{N-3}4}.\nonumber
\end{gather}
Taking into account that $-\ln \alpha_2 = -P^2 \ln \alpha_1$ we obtain that
\begin{gather}
e^{\sqrt{x_1^*x_2^*}} \sim \( \Gamma\(\frac{N-1}2\) \alpha_1\)^{-2P} \cdot \((- \ln \alpha_1)^{P} \cdot (-\ln \alpha_2)^{\frac{1}{P}}\)^{\frac{N-3}2}= \nonumber \\
= \( \Gamma\(\frac{N-1}2\) \alpha_1\)^{-2P} \cdot \((- \ln \alpha_1)^{P} \cdot ( -P^2 \ln \alpha_1)^{\frac{1}{P}}\)^{\frac{N-3}2} = \nonumber \\
=\( \Gamma\(\frac{N-1}2\) \alpha_1\)^{-2P}  \cdot P^{(\frac{N-3}P)} \cdot \(- \ln \alpha_1 \)^{(P+\frac1{P})\frac{N-3}2}. \label{5STATexpx1x2}
\end{gather}
Thus, taking into account (\ref{5STATexpx}), (\ref{5STATalph}) and (\ref{5STATexpx1x2}) we get that
\begin{gather}
\alpha \sim \frac{c \cdot {(P  \cdot \frac{-\ln \alpha_1}{1-c^2}  )}^{(\delta+\frac{3}2 )} {\tilde{Q}}^{-\frac1{2(1-c^2)}}}{2(K_2  c)^{1+\delta}\Gamma(1+\delta)\sqrt{\pi c}{\( \frac{-\ln \alpha_1}{1-c^2} \)}^{2} P (P-c)(1-c P)} \label{5STATpredposl} \\
\text{where \ \ } \tilde{Q}:=e^{x_1^*} \cdot e^{x_2^*}  \cdot e^{-2c\sqrt{x_1^* x_2^*}} \sim \nonumber \\
\sim \(-\ln \alpha_1\)^{N-3} \cdot \(\alpha_1\Gamma\(\frac{N-1}2\)\)^{-2} \cdot \(-\ln \alpha_2\)^{N-3} \cdot \(\alpha_2\Gamma\(\frac{N-1}2\)\)^{-2} \cdot e^{-2c\sqrt{x_1^* x_2^*}} =\nonumber \\
= \(\prod_{i=1}^2 \frac{\ln^{N-3} \alpha_i}{\alpha_i^2 } \) \cdot \(\Gamma\(\frac{N-1}2\)\)^{-4}
\cdot \(    \( \Gamma\(\frac{N-1}2\) \alpha_1\)^{-2P}  \cdot P^{\frac{N-3}P} \cdot \(- \ln \alpha_1 \)^{(P+\frac1{P})\frac{N-3}2}           \)^{-2c} = \nonumber \\
=\frac{(\ln^{N-3} \alpha_1) (\ln^{N-3} \alpha_2) }{ \alpha_1^2 \alpha_2^2 }   \cdot \( \Gamma\(\frac{N-1}2\)\)^{4Pc-4} \cdot \alpha_1^{4Pc} \cdot P^{2(\frac{c}{P})(3-N)}
\cdot \(- \ln \alpha_1 \)^{(P+\frac1{P})(3-N)c}    =: Q      . \nonumber
\end{gather}
Since $-\ln \alpha_2 = -P^2 \ln \alpha_1$ we have $\alpha_2 = \alpha_1^{P^2}$ whence
\begin{gather}
Q=\frac{(-\ln \alpha_1)^{N-3} (-P^2 \ln \alpha_1)^{N-3} }{ \alpha_1^2 \alpha_1^{2P^2} }   \cdot \( \Gamma\(\frac{N-1}2\)\)^{4Pc-4} \cdot \alpha_1^{4Pc} \cdot P^{2(\frac{c}{P})(3-N)}
\cdot \(- \ln \alpha_1 \)^{(P+\frac1{P})(3-N)c}    = \nonumber \\
=\frac{P^{2(N-3)(1-\frac{c}{P})} }{ \(\Gamma\(\frac{N-1}2\) \)^{4(1-Pc)} } \cdot (-\ln \alpha_1)^{(N-3)(2-c(P+P^{-1}))} \cdot  \alpha_1^{-2(P^2-2Pc+1)} .
 \nonumber
\end{gather}
In the formula (\ref{5STATpredposl}) replace  $\delta$ and $K_2$ with $\frac{N-3}2$ and $\frac1{1-c^2}$ respectively. We get 
\begin{gather}
\alpha \sim \frac{c \cdot {(P  \cdot \frac{-\ln \alpha_1}{1-c^2}  )}^{\frac{N}2 } Q^{-\frac1{2(1-c^2)}}}{2(\frac{c}{1-c^2})^{\frac{N-1}2}\Gamma(\frac{N-1}2)\sqrt{c\pi}{\( \frac{-\ln \alpha_1}{1-c^2} \)}^{2} P (P-c)(1-c P)}       =\nonumber \\
=\frac{(1-c^2)^{\frac{N-1}2}\cdot P^{\frac{N}2-1}\cdot (- \frac{\ln \alpha_1}{1-c^2})^{\frac{N}2-2} \cdot Q^{-\frac1{2(1-c^2)}}}{2c^{\frac{N}2-1}\sqrt{\pi}\Gamma\(\frac{N-1}2\)(P-c)(1-c P)} \nonumber
=\frac{(1-c^2)^{\frac{3}2}\cdot P^{\frac{N}2-1}\cdot (- \ln \alpha_1)^{\frac{N}2-2} \cdot Q^{-\frac1{2(1-c^2)}}}{2c^{\frac{N}2-1}\sqrt{\pi}\Gamma\(\frac{N-1}2\)(P-c)(1-c P)}    .\nonumber
\end{gather}
Thus, Theorem \ref{5STATth3} is proved.\\
Let us prove Corollary \ref{5STATsled1}. In fact, suppose the conditions of Theorem \ref{5STATth3} hold true and suppose $\alpha_1 = \alpha_2$. Then $P=1$. Hence,
\begin{gather}
Q = \(\frac{(-\ln \alpha_1)^{\frac{N-3}2}}{\alpha_1 \cdot \Gamma(\frac{N-1}2)}\)^{4(1-c)}. \nonumber
\end{gather}
In this case \begin{gather}
\alpha \sim \frac{(1-c^2)^{\frac{3}2}\cdot (- \ln \alpha_1)^{\frac{N}2-2} \cdot Q^{-\frac1{2(1-c^2)}}}{2c^{\frac{N}2-1}\sqrt{\pi}\Gamma\(\frac{N-1}2\)(1-c)^2}  =D \cdot {\(-\ln \alpha_1 \)}^{\frac{N}2-2} Q^{-\frac1{2(1-c^2)}} \nonumber
\end{gather}
where $D:=\frac{(1-c^2)^{\frac{3}2}}{2c^{\frac{N}2-1}\sqrt{\pi}\Gamma\(\frac{N-1}2\)(1-c)^2} $. So,
\begin{gather}
\alpha \sim D \cdot {\(-\ln \alpha_1  \)}^{\frac{N}2-2} \(\frac{(-\ln \alpha_1)^{\frac{N-3}2}}{\alpha_1 \cdot \Gamma(\frac{N-1}2)}\)^{\frac{-2}{1+c}}=D \cdot {\(-\ln \alpha_1 \)}^{\frac{N}2-2} \(\frac{\alpha_1 \cdot \Gamma(\frac{N-1}2)}{(-\ln \alpha_1)^{\frac{N-3}2}}\)^{\frac{2}{1+c}}.\nonumber
\end{gather}
Hence,
\begin{gather}
\alpha \sim \frac{(1-c^2)^{\frac{3}2}}{2c^{\frac{N}2-1}\sqrt{\pi}\Gamma\(\frac{N-1}2\)(1-c)^2}\cdot {\(-\ln \alpha_1 \)}^{\frac{N}2-2} \(\frac{\alpha_1 \cdot \Gamma(\frac{N-1}2)}{(-\ln \alpha_1)^{\frac{N-3}2}}\)^{\frac{2}{1+c}}= \nonumber\\
=\frac{(1-c^2)^{\frac32}{\(\Gamma(\frac{N-1}2)\)}^{\frac{1-c}{1+c}}}{2c^{\frac{N}2-1}\sqrt{\pi}(1-c)^2} \cdot
\(- \ln \alpha_1\)^{\frac{N}2 - \frac{N-3}{1+c}-2} \(\alpha_1 \)^{\frac{2}{1+c}}
\nonumber
\end{gather}
and $\frac{2}{1+c} \in (1,2)$ since $c \in (0,1)$. Thus, Corollary \ref{5STATsled1} is proved.

\section{Proof of Theorem \ref{5STATth4}}
Suppose $d \ge 2$ and $Bes_d(t)$ is the $d$-dimensional Bessel process. By definition, put $N:=d+1$. It was shown in \cite{Zubkov i ya} (see Theorem 2 and formula (11)) that if $0 < t_1 < t_2 \le 1$ and $x_1, x_2 \ge 0$ then
\begin{gather}
P(Bes_{N-1}(t_1) \ge x_1, ..., Bes_{N-1}(t_k) \ge x_k) = \lim_{n \to \infty} P\(X({\lfloor n t_1 \rfloor}) \ge \frac{x_1^2}{t_1},..., X({\lfloor n t_k \rfloor}) \ge \frac{x_k^2}{t_k}\).\label{5STATflabessel1}
\end{gather}
However,
\begin{gather}
P(X(n_1) > x_1^*, ..., X(n_k) > x_k^*) = \alpha(x_1^*, ..., x_k^*).
 \label{5STATflabessel2}
\end{gather}
Further, it is obvious that if  $c \ge 0$ then
\begin{gather}
Bes_d(ct) \overset{\text{  d }}{=} \sqrt{c} Bes_d (t). \label{5STATflabessel3}
\end{gather}
We fix the numbers $0 < s_1 < s_2$ and $x_1, x_2 \ge 0$.
Let $\tilde{s}:=\frac{s_1}{s_2}$, $\tilde{x}_1: = {\frac{x_1}{\sqrt{s_2}}}$,  $\tilde{x}_2 := {\frac{x_2}{\sqrt{s_2}}}$, $x_1^*:=\frac{\tilde{x}_1^2}{\tilde{s}}$, $x_2^*:=\tilde{x}_2^2$.
In view of (\ref{5STATflabessel1})--(\ref{5STATflabessel3}) we have:
\begin{gather}
P(Bes_{d}(s_1) \ge x_1, Bes_{d}(s_2) \ge x_2) = P(Bes_{d}(s_2 \cdot \tilde{s}) \ge x_1, Bes_{d}(s_2 \cdot 1) \ge x_2)= \\
=P(\sqrt{s_2} Bes_{d}(\tilde{s}) \ge x_1, \sqrt{s_2}Bes_{d}(1) \ge x_2)  = P(Bes_{d}(\tilde{s}) \ge \tilde{x}_1, Bes_{d}(1) \ge \tilde{x}_2) = \nonumber \\
=\lim_{n \to \infty} P\(X({\lfloor n \tilde{s} \rfloor}) \ge x_1^*, X({\lfloor n \cdot 1 \rfloor}) \ge x_2^*\) = \alpha(x_1^*, x_2^*) \label{5STATlastfla}
\end{gather}
for $n_1:=\lfloor n \tilde{s} \rfloor$, $n_2:=n$. Further, $c:=\lim_{n_1, n_2 \to + \infty} \sqrt{\frac{n_1}{n_2}} = \sqrt{\tilde{s}}$.
Let $x_1, x_2 \to + \infty$ in a such a way that $\rho = \sqrt{\frac{x_2^*}{x_1^*}} = const$ and $c < \rho < \frac1{c}$ (which is equivalent to the fact that $\sqrt{\frac{s_1}{s_2}} < \frac{x_2}{x_1} \cdot \sqrt{\frac{s_1}{s_2}} < \sqrt{\frac{s_2}{s_1}}$). Then Theorem \ref{5STATth2} is applicable (since $N = d +1 \ge 3$) and this fact together with (\ref{5STATlastfla}) (and equality $N = d+1$) completes the proof of Theorem \ref{5STATth4}.
\begin{remark}
The asymptotic properties of tails of the form $P(Bes_{d}(t_1) \ge x_1, ..., Bes_{d}(t_k) \ge x_k)$ for the case $k > 2$ may be obtained with the help of Theorem \ref{5STATth4} and the Bonferroni inequalities.
\end{remark}


{}
\end{document}